\documentclass{amsart}

\usepackage{amssymb,amsmath,amsthm,amscd,epsf,latexsym,verbatim,graphicx,amsfonts,mathtools} 
\usepackage{adjustbox, subfig}
\usepackage{csquotes}
\usepackage[english]{babel}
\usepackage[shortlabels]{enumitem}
\usepackage{mathrsfs}
\usepackage{tikz-cd}
\usepackage{multirow}
\usepackage[pdfencoding=auto, psdextra]{hyperref}
\usepackage{array}

\usepackage{caption}
\usepackage[normalem]{ulem}

\newcounter{Theorem}[section]
\theoremstyle{definition}

\newtheorem{definition}[Theorem]{Definition}

\theoremstyle{plain}
\newtheorem{theorem}[Theorem]{Theorem}

\newtheorem{lemma}[Theorem]{Lemma}
\numberwithin{Theorem}{section}

\newtheorem{problem}[Theorem]{Problem}

\newtheorem{rthm}{Theorem}
\newenvironment{reptheorem}[1]
  {\renewcommand\therthm{\ref*{#1}}\rthm}
  {\endrthm}

\newcommand{\D}{\mathcal{D}}
\newcommand{\Dm}{\mathcal{D}^{m}}
\newcommand{\T}{\mathcal{T}}

\newcommand{\size}[1]{\left\langle #1\right\rangle}

\usepackage{todonotes}

\newcommand{\old}[1]{{}}

\tikzset
{
  tree5.1/.pic=
  {%
    \draw (0,0) -- (0.5,0.5) -- (1,0);
    \draw (0.5,0.5) -- (1,1) -- (2,0);
    \draw (1,1) -- (1.5,1.5) -- (3,0);
    \draw (1.5,1.5) -- (2,2) -- (4,0);
    \node[fill=black,rectangle,inner sep=2pt]  at (0,0) {};      
    \node[fill=black,rectangle,inner sep=2pt]  at (1,0) {};      
    \node[fill=black,rectangle,inner sep=2pt]  at (2,0) {}; 
    \node[fill=black,rectangle,inner sep=2pt]  at (3,0) {}; 
    \node[fill=black,rectangle,inner sep=2pt]  at (4,0) {}; 
    \node[fill=black,circle,inner sep=1pt]  at (0.5,0.5) {};
    \node[fill=black,circle,inner sep=1pt]  at (1,1) {};
    \node[fill=black,circle,inner sep=1pt]  at (1.5,1.5) {};
    \node[fill=black,circle,inner sep=1pt]  at (2,2) {};
  }
}

\tikzset
{
  tree5.2/.pic=
  {%
    \draw (0,0) -- (0.5,0.5) -- (1,0);
    \draw (0.5,0.5) -- (1.5,1.5) -- (2.5,0.5);
    \draw (2,0) -- (2.5,0.5) -- (3,0);
    \draw (1.5,1.5) -- (2,2) -- (4,0);
    \node[fill=black,rectangle,inner sep=2pt]  at (0,0) {};      
    \node[fill=black,rectangle,inner sep=2pt]  at (1,0) {};      
    \node[fill=black,rectangle,inner sep=2pt]  at (2,0) {}; 
    \node[fill=black,rectangle,inner sep=2pt]  at (3,0) {}; 
    \node[fill=black,rectangle,inner sep=2pt]  at (4,0) {}; 
    \node[fill=black,circle,inner sep=1pt]  at (0.5,0.5) {};
    \node[fill=black,circle,inner sep=1pt]  at (2,2) {};
    \node[fill=black,circle,inner sep=1pt]  at (1.5,1.5) {};
    \node[fill=black,circle,inner sep=1pt]  at (2.5,0.5) {};
  }
}

\tikzset
{
  tree5.3/.pic=
  {%
    \draw (0,0) -- (0.5,0.5) -- (1,0);
    \draw (0.5,0.5) -- (1,1) -- (2,0);
    \draw (3,0) -- (3.5,0.5) -- (4,0);
    \draw (1,1) -- (2,2) -- (3.5,0.5);
    \node[fill=black,rectangle,inner sep=2pt]  at (0,0) {};      
    \node[fill=black,rectangle,inner sep=2pt]  at (1,0) {};      
    \node[fill=black,rectangle,inner sep=2pt]  at (2,0) {}; 
    \node[fill=black,rectangle,inner sep=2pt]  at (3,0) {}; 
    \node[fill=black,rectangle,inner sep=2pt]  at (4,0) {}; 
    \node[fill=black,circle,inner sep=1pt]  at (0.5,0.5) {};
    \node[fill=black,circle,inner sep=1pt]  at (1,1) {};
    \node[fill=black,circle,inner sep=1pt]  at (2,2) {};
    \node[fill=black,circle,inner sep=1pt]  at (3.5,0.5) {};
  }
}

\tikzset
{
both_caterpillars/.pic=
{%

\foreach \x in {1,2,4,6,7,8}{ \node[fill=black,circle,inner sep=1pt ] (n\x) at (-\x/2-3,\x-8) {};  }
\draw (-4-3,8-8) -- (-3,-8); 

\node[fill=black, rectangle, inner sep=1.5pt, label = right:{$v_{1}$}] at (-3,8) {};
\draw (-4-3,8-8) -- (-3,2*8-8);
\node[fill=black, rectangle, inner sep=1.5pt, label = right:{$v_{2}$}] at (-3,6) {};
\draw (-3.5-3,7-8) -- (-3,2*7-8);
\node[fill=black, rectangle, inner sep=1.5pt, label = right:{$v_{3}$}] at (-3,4) {};
\draw (-3-3,6-8) -- (-3,2*6-8);
\node[label = center:{$\vdots$}] at (-3,2) {};
\node[fill=black, rectangle, inner sep=1.5pt, label =  {[xshift= 4mm, yshift=-5mm] $v_{m}$ }] at (-3,0) {};
\draw (-2-3,4-8) -- (-3,2*4-8);
\node[label = center:{$\vdots$}] at (-3,-2) {};
\node[fill=black, rectangle, inner sep=1.5pt, label = right:{$v_{n-2}$}] at (-3,-4) {};
\draw (-1-3,2-8) -- (-3,2*2-8);
\node[fill=black, rectangle, inner sep=1.5pt, label = right:{$v_{n-1}$}] at (-3,-6) {};
\draw (-0.5-3,1-8) -- (-3,2*1-8);
\node[fill=black, rectangle, inner sep=1.5pt, label = right:{$v_{n}$}] at (-3,-8) {};

\foreach \x in {1,2,4,6,7,8}{ \node[fill=black,circle,inner sep=1pt, label = left:{} ] (n\x) at ( \x/2 + 3, \x - 8 ) {}; }
\draw (4+3,8-8) -- (3,-8); 

\node[fill=black, rectangle, inner sep=1.5pt, label = left:{$u_{1}$}] at (3,8) {};
\draw (4+3,8-8) -- (3,2*8-8);
\node[fill=black, rectangle, inner sep=1.5pt, label = left:{$u_{2}$}] at (3,6) {};
\draw (3.5+3,7-8) -- (3,2*7-8);
\node[fill=black, rectangle, inner sep=1.5pt, label = left:{$u_{3}$}] at (3,4) {};
\draw (3+3,6-8) -- (3,2*6-8);
\node[label = center:{$\vdots$}] at (3,2) {};
\node[fill=black, rectangle, inner sep=1.5pt, label = {[xshift=-1.5mm, yshift=-5mm] $u_{n-m+1}$ } ] at (3,0) {};
\draw (2+3,4-8) -- (3,2*4-8);
\node[label = center:{$\vdots$}] at (3,-2) {};
\node[fill=black, rectangle, inner sep=1.5pt, label = left:{$u_{n-2}$}] at (3,-4) {};
\draw (1+3,2-8) -- (3,2*2-8);
\node[fill=black, rectangle, inner sep=1.5pt, label = left:{$u_{n-1}$}] at (3,-6) {};
\draw (0.5 + 3,1-8) -- (3,2*1-8);
\node[fill=black, rectangle, inner sep=1.5pt, label = left:{$u_{n}$}] at (3,-8) {};

\draw[dashed] (-3,8) -- (3,-8);
\draw[dashed] (-3,6) -- (3,-6);
\draw[dashed] (-3,4) -- (3,-4);
\draw[dashed] (-3,0) -- (3,0);

}
}

\tikzset
{
both_caterpillars_card_1/.pic=
{%

\foreach \x in {1,3,5,6}{ \node[fill=black,circle,inner sep=1pt ] (n\x) at (-\x/2-3,\x-6) {};  }
\draw (-6,0) -- (-3,-6); 

\node[fill=black, rectangle, inner sep=1.5pt, label = right:{$v'_{1}$}] at (-3,6) {};
\draw (-3-3,6-6) -- (-3,2*7-8);
\node[fill=black, rectangle, inner sep=1.5pt, label = right:{$v'_{2}$}] at (-3,4) {};
\draw (-2.5-3,5-6) -- (-3,2*6-8);
\node[label = center:{$\vdots$}] at (-3,2) {};
\node[fill=black, rectangle, inner sep=1.5pt, label =  {[xshift= 5mm, yshift=-6mm] $v'_{m-1}$ }] at (-3,0) {};
\draw (-1.5-3,3-6) -- (-3,2*4-8);
\node[label = center:{$\vdots$}] at (-3,-2) {};
\node[fill=black, rectangle, inner sep=1.5pt, label = right:{$v'_{n-2}$}] at (-3,-4) {};
\draw (-0.5-3,1-6) -- (-3,2*2-8);
\node[fill=black, rectangle, inner sep=1.5pt, label = right:{$v'_{n-1}$}] at (-3,-6) {};

\foreach \x in {1,3,5,6}{ \node[fill=black,circle,inner sep=1pt ] (n\x) at (\x/2+3,\x-6) {};  }
\draw (3+3,0) -- (3,-6); 

\node[fill=black, rectangle, inner sep=1.5pt, label = left:{$u'_{1}$}] at (3,6) {};
\draw (3+3,6-6) -- (3,2*7-8);
\node[fill=black, rectangle, inner sep=1.5pt, label = left:{$u'_{2}$}] at (3,4) {};
\draw (2.5+3,5-6) -- (3,2*6-8);
\node[label = center:{$\vdots$}] at (3,2) {};
\node[fill=black, rectangle, inner sep=1.5pt, label = {[xshift= -2mm, yshift=-6mm] $u'_{n-m+1}$ } ] at (3,0) {};
\draw (1.5+3,3-6) -- (3,2*4-8);
\node[label = center:{$\vdots$}] at (3,-2) {};
\node[fill=black, rectangle, inner sep=1.5pt, label = left:{$u'_{n-2}$}] at (3,-4) {};
\draw (0.5+3,1-6) -- (3,2*2-8);
\node[fill=black, rectangle, inner sep=1.5pt, label = left:{$u'_{n-1}$}] at (3,-6) {};

\draw[dashed] (-3,6) -- (3,-6);
\draw[dashed] (-3,4) -- (3,-4);
\draw[dashed] (-3,0) -- (3,0);

}
}

\tikzset
{
both_caterpillars_old/.pic=
{%

\foreach \x in {1,2,4,6,7,8}{ \node[fill=black,circle,inner sep=1pt ] (n\x) at (-\x-3,\x-8) {};  }
\draw (-11,0) -- (-3,-8); 

\node[fill=black, rectangle, inner sep=1.5pt, label = right:{$v_{1}$}] at (-3,8) {};
\draw (-8-3,8-8) -- (-3,2*8-8);
\node[fill=black, rectangle, inner sep=1.5pt, label = right:{$v_{2}$}] at (-3,6) {};
\draw (-7-3,7-8) -- (-3,2*7-8);
\node[fill=black, rectangle, inner sep=1.5pt, label = right:{$v_{3}$}] at (-3,4) {};
\draw (-6-3,6-8) -- (-3,2*6-8);
\node[label = center:{$\vdots$}] at (-3,2) {};
\node[fill=black, rectangle, inner sep=1.5pt, label =  {[xshift= 4mm, yshift=-5mm] $v_{m}$ }] at (-3,0) {};
\draw (-4-3,4-8) -- (-3,2*4-8);
\node[label = center:{$\vdots$}] at (-3,-2) {};
\node[fill=black, rectangle, inner sep=1.5pt, label = right:{$v_{n-2}$}] at (-3,-4) {};
\draw (-2-3,2-8) -- (-3,2*2-8);
\node[fill=black, rectangle, inner sep=1.5pt, label = right:{$v_{n-1}$}] at (-3,-6) {};
\draw (-1-3,1-8) -- (-3,2*1-8);
\node[fill=black, rectangle, inner sep=1.5pt, label = right:{$v_{n}$}] at (-3,-8) {};

\foreach \x in {1,2,4,6,7,8}{ \node[fill=black,circle,inner sep=1pt, label = left:{} ] (n\x) at (\x+3,\x-8) {}; }
\draw (11,0) -- (3,-8); 

\node[fill=black, rectangle, inner sep=1.5pt, label = left:{$u_{1}$}] at (3,8) {};
\draw (8+3,8-8) -- (3,2*8-8);
\node[fill=black, rectangle, inner sep=1.5pt, label = left:{$u_{2}$}] at (3,6) {};
\draw (7+3,7-8) -- (3,2*7-8);
\node[fill=black, rectangle, inner sep=1.5pt, label = left:{$u_{3}$}] at (3,4) {};
\draw (6+3,6-8) -- (3,2*6-8);
\node[label = center:{$\vdots$}] at (3,2) {};
\node[fill=black, rectangle, inner sep=1.5pt, label = {[xshift=-1.5mm, yshift=-5mm] $u_{n-m+1}$ } ] at (3,0) {};
\draw (4+3,4-8) -- (3,2*4-8);
\node[label = center:{$\vdots$}] at (3,-2) {};
\node[fill=black, rectangle, inner sep=1.5pt, label = left:{$u_{n-2}$}] at (3,-4) {};
\draw (2+3,2-8) -- (3,2*2-8);
\node[fill=black, rectangle, inner sep=1.5pt, label = left:{$u_{n-1}$}] at (3,-6) {};
\draw (1+3,1-8) -- (3,2*1-8);
\node[fill=black, rectangle, inner sep=1.5pt, label = left:{$u_{n}$}] at (3,-8) {};

\draw[dashed] (-3,8) -- (3,-8);
\draw[dashed] (-3,6) -- (3,-6);
\draw[dashed] (-3,4) -- (3,-4);
\draw[dashed] (-3,0) -- (3,0);

}
}

\tikzset
{
both_caterpillars_card_1_old/.pic=
{%

\foreach \x in {1,3,5,6}{ \node[fill=black,circle,inner sep=1pt ] (n\x) at (-\x-3,\x-6) {};  }
\draw (-9,0) -- (-3,-6); 

\node[fill=black, rectangle, inner sep=1.5pt, label = right:{$v'_{1}$}] at (-3,6) {};
\draw (-6-3,6-6) -- (-3,2*7-8);
\node[fill=black, rectangle, inner sep=1.5pt, label = right:{$v'_{2}$}] at (-3,4) {};
\draw (-5-3,5-6) -- (-3,2*6-8);
\node[label = center:{$\vdots$}] at (-3,2) {};
\node[fill=black, rectangle, inner sep=1.5pt, label =  {[xshift= 5mm, yshift=-6mm] $v'_{m-1}$ }] at (-3,0) {};
\draw (-3-3,3-6) -- (-3,2*4-8);
\node[label = center:{$\vdots$}] at (-3,-2) {};
\node[fill=black, rectangle, inner sep=1.5pt, label = right:{$v'_{n-2}$}] at (-3,-4) {};
\draw (-1-3,1-6) -- (-3,2*2-8);
\node[fill=black, rectangle, inner sep=1.5pt, label = right:{$v'_{n-1}$}] at (-3,-6) {};

\foreach \x in {1,3,5,6}{ \node[fill=black,circle,inner sep=1pt ] (n\x) at (\x+3,\x-6) {};  }
\draw (9,0) -- (3,-6); 

\node[fill=black, rectangle, inner sep=1.5pt, label = left:{$u'_{1}$}] at (3,6) {};
\draw (6+3,6-6) -- (3,2*7-8);
\node[fill=black, rectangle, inner sep=1.5pt, label = left:{$u'_{2}$}] at (3,4) {};
\draw (5+3,5-6) -- (3,2*6-8);
\node[label = center:{$\vdots$}] at (3,2) {};
\node[fill=black, rectangle, inner sep=1.5pt, label = {[xshift= -2mm, yshift=-6mm] $u'_{n-m+1}$ } ] at (3,0) {};
\draw (3+3,3-6) -- (3,2*4-8);
\node[label = center:{$\vdots$}] at (3,-2) {};
\node[fill=black, rectangle, inner sep=1.5pt, label = left:{$u'_{n-2}$}] at (3,-4) {};
\draw (1+3,1-6) -- (3,2*2-8);
\node[fill=black, rectangle, inner sep=1.5pt, label = left:{$u'_{n-1}$}] at (3,-6) {};

\draw[dashed] (-3,6) -- (3,-6);
\draw[dashed] (-3,4) -- (3,-4);
\draw[dashed] (-3,0) -- (3,0);

}
}

\tikzset
{
example_5_1/.pic=
{%
\node[fill=black,rectangle,inner sep=2pt, label = right:{$v_1$}]  at (-1.5,2) {}; 
\node[fill=black,rectangle,inner sep=2pt, label = right:{$v_2$}]  at (-1.5,1) {}; 
\node[fill=black,rectangle,inner sep=2pt, label = right:{$v_3$}]  at (-1.5,0) {}; 
\node[fill=black,rectangle,inner sep=2pt, label = right:{$v_4$}]  at (-1.5,-1) {}; 
\node[fill=black,rectangle,inner sep=2pt, label = right:{$v_5$}]  at (-1.5,-2) {}; 
\node[fill=black,rectangle,inner sep=2pt, label = left:{$u_1$}]  at (1.5,2) {}; 
\node[fill=black,rectangle,inner sep=2pt, label = left:{$u_2$}]  at (1.5,1) {}; 
\node[fill=black,rectangle,inner sep=2pt, label = left:{$u_3$}]  at (1.5,0) {}; 
\node[fill=black,rectangle,inner sep=2pt, label = left:{$u_4$}]  at (1.5,-1) {}; 
\node[fill=black,rectangle,inner sep=2pt, label = left:{$u_5$}]  at (1.5,-2) {}; 
\draw[dashed] (-1.5,2)  -- (1.5,-1);
\draw[dashed] (-1.5,1)  -- (1.5,-2);
\draw[dashed] (-1.5,0)  -- (1.5,1);
\draw[dashed] (-1.5,-1) -- (1.5,2);
\draw[dashed] (-1.5,-2) -- (1.5,0);
\draw (-1.5,2) -- (-2,1.5) -- (-2,1.5) -- (-2.5,1) -- (-2.5,1) -- (-3.5,0) -- (-1.5,-2) -- (-2,-1.5) -- (-1.5,-1);
\draw (-1.5,0) -- (-2.5,-1);
\draw (-1.5,1) -- (-3,-0.5);
\draw (1.5,2) -- (2,1.5) -- (2,1.5) -- (2.5,1) -- (2.5,1) -- (3.5,0) -- (1.5,-2) -- (2,-1.5) -- (1.5,-1);
\draw (1.5,0) -- (2.5,-1);
\draw (1.5,1) -- (3,-0.5);
\node[fill=black,circle,inner sep=1pt]  at (-2,-1.5) {};
\node[fill=black,circle,inner sep=1pt]  at (-2.5,-1) {};
\node[fill=black,circle,inner sep=1pt]  at (-3,-0.5) {};
\node[fill=black,circle,inner sep=1pt]  at (-3.5,0) {};
\node[fill=black,circle,inner sep=1pt]  at (2,-1.5) {};
\node[fill=black,circle,inner sep=1pt]  at (2.5,-1) {};
\node[fill=black,circle,inner sep=1pt]  at (3,-0.5) {};
\node[fill=black,circle,inner sep=1pt]  at (3.5,0) {};
\node at (0,-3) {\huge{$\T_1$}};
}
}

\tikzset
{
example_5_2/.pic=
{%
\node[fill=black,rectangle,inner sep=2pt, label = right:{$v_1$}]  at (-1.5,2) {}; 
\node[fill=black,rectangle,inner sep=2pt, label = right:{$v_2$}]  at (-1.5,1) {}; 
\node[fill=black,rectangle,inner sep=2pt, label = right:{$v_3$}]  at (-1.5,0) {}; 
\node[fill=black,rectangle,inner sep=2pt, label = right:{$v_4$}]  at (-1.5,-1) {}; 
\node[fill=black,rectangle,inner sep=2pt, label = right:{$v_5$}]  at (-1.5,-2) {}; 
\node[fill=black,rectangle,inner sep=2pt, label = left:{$u_1$}]  at (1.5,2) {}; 
\node[fill=black,rectangle,inner sep=2pt, label = left:{$u_2$}]  at (1.5,1) {}; 
\node[fill=black,rectangle,inner sep=2pt, label = left:{$u_3$}]  at (1.5,0) {}; 
\node[fill=black,rectangle,inner sep=2pt, label = left:{$u_4$}]  at (1.5,-1) {}; 
\node[fill=black,rectangle,inner sep=2pt, label = left:{$u_5$}]  at (1.5,-2) {}; 
\draw[dashed] (-1.5,2) -- (1.5,-1);
\draw[dashed] (-1.5,1) -- (1.5,0);
\draw[dashed] (-1.5,0) -- (1.5,-2);
\draw[dashed] (-1.5,-1) -- (1.5,2);
\draw[dashed] (-1.5,-2) -- (1.5,1);
\draw (-1.5,2) -- (-2,1.5) -- (-2,1.5) -- (-2.5,1) -- (-2.5,1) -- (-3.5,0) -- (-1.5,-2) -- (-2,-1.5) -- (-1.5,-1);
\draw (-1.5,0) -- (-2.5,-1);
\draw (-1.5,1) -- (-3,-0.5);
\draw (1.5,2) -- (2,1.5) -- (2,1.5) -- (2.5,1) -- (2.5,1) -- (3.5,0) -- (1.5,-2) -- (2,-1.5) -- (1.5,-1);
\draw (1.5,0) -- (2.5,-1);
\draw (1.5,1) -- (3,-0.5);
\node[fill=black,circle,inner sep=1pt]  at (-2,-1.5) {};
\node[fill=black,circle,inner sep=1pt]  at (-2.5,-1) {};
\node[fill=black,circle,inner sep=1pt]  at (-3,-0.5) {};
\node[fill=black,circle,inner sep=1pt]  at (-3.5,0) {};
\node[fill=black,circle,inner sep=1pt]  at (2,-1.5) {};
\node[fill=black,circle,inner sep=1pt]  at (2.5,-1) {};
\node[fill=black,circle,inner sep=1pt]  at (3,-0.5) {};
\node[fill=black,circle,inner sep=1pt]  at (3.5,0) {};
\node at (0,-3) {\huge{$\T_2$}};
}
}

\tikzset
{
example_5_card_1/.pic=
{%
\node[fill=black,rectangle,inner sep=2pt]  at (-1,1.5) {}; 
\node[fill=black,rectangle,inner sep=2pt]  at (-1,0.5) {}; 
\node[fill=black,rectangle,inner sep=2pt]  at (-1,-0.5) {}; 
\node[fill=black,rectangle,inner sep=2pt]  at (-1,-1.5) {}; 
\node[fill=black,rectangle,inner sep=2pt]  at (1,1.5) {}; 
\node[fill=black,rectangle,inner sep=2pt]  at (1,0.5) {}; 
\node[fill=black,rectangle,inner sep=2pt]  at (1,-0.5) {}; 
\node[fill=black,rectangle,inner sep=2pt]  at (1,-1.5) {}; 
\draw[dashed] (-1,1.5) -- (1,-0.5);
\draw[dashed] (-1,0.5) -- (1,-1.5);
\draw[dashed] (-1,-0.5) -- (1,1.5);
\draw[dashed] (-1,-1.5) -- (1,0.5);
\draw (-1,-1.5) -- (-2.5,0);
\draw (-1,-0.5) -- (-1.5,-1);
\draw (-1,0.5) -- (-2,-0.5);
\draw (-1,1.5) -- (-2.5,0);
\draw (1,-1.5) -- (2.5,0);
\draw (1,-0.5) -- (1.5,-1);
\draw (1,0.5) -- (2,-0.5);
\draw (1,1.5) -- (2.5,0);
\node[fill=black,circle,inner sep=1pt]  at (-1.5,-1) {};
\node[fill=black,circle,inner sep=1pt]  at (-2,-0.5) {};
\node[fill=black,circle,inner sep=1pt]  at (-2.5,0) {};
\node[fill=black,circle,inner sep=1pt]  at (1.5,-1) {};
\node[fill=black,circle,inner sep=1pt]  at (2,-0.5) {};
\node[fill=black,circle,inner sep=1pt]  at (2.5,0) {};
\node at (0,-2) {$\mathcal{T}_1 - v_3u_2=\T_1-v_4u_1=\T_1-v_5u_3$};
\node at (0,-2.5) {$\T_2-v_1u_4=\T_2-v_2u_3=\T_2-v_3u_5$};
}
}

\tikzset
{
example_5_card_2/.pic=
{%
\node[fill=black,rectangle,inner sep=2pt]  at (-1,1.5) {}; 
\node[fill=black,rectangle,inner sep=2pt]  at (-1,0.5) {}; 
\node[fill=black,rectangle,inner sep=2pt]  at (-1,-0.5) {}; 
\node[fill=black,rectangle,inner sep=2pt]  at (-1,-1.5) {}; 
\node[fill=black,rectangle,inner sep=2pt]  at (1,1.5) {}; 
\node[fill=black,rectangle,inner sep=2pt]  at (1,0.5) {}; 
\node[fill=black,rectangle,inner sep=2pt]  at (1,-0.5) {}; 
\node[fill=black,rectangle,inner sep=2pt]  at (1,-1.5) {}; 
\draw[dashed] (-1,1.5) -- (1,-0.5);
\draw[dashed] (-1,0.5) -- (1, 0.5);
\draw[dashed] (-1,-0.5) -- (1,1.5);
\draw[dashed] (-1,-1.5) -- (1,-1.5);
\draw (-1,-1.5) -- (-2.5,0);
\draw (-1,-0.5) -- (-1.5,-1);
\draw (-1,0.5) -- (-2,-0.5);
\draw (-1,1.5) -- (-2.5,0);
\draw (1,-1.5) -- (2.5,0);
\draw (1,-0.5) -- (1.5,-1);
\draw (1,0.5) -- (2,-0.5);
\draw (1,1.5) -- (2.5,0);
\node[fill=black,circle,inner sep=1pt]  at (-1.5,-1) {};
\node[fill=black,circle,inner sep=1pt]  at (-2,-0.5) {};
\node[fill=black,circle,inner sep=1pt]  at (-2.5,0) {};
\node[fill=black,circle,inner sep=1pt]  at (1.5,-1) {};
\node[fill=black,circle,inner sep=1pt]  at (2,-0.5) {};
\node[fill=black,circle,inner sep=1pt]  at (2.5,0) {};
\node at (0,-2) {$\mathcal{T}_1 - v_1u_4=\T_1-v_2u_5$};
\node at (0,-2.5) {$\T_2-v_4u_1=\T_2-v_5u_3$};
}
}

\tikzset
{
example_6_1/.pic=
{%
\node[fill=black,rectangle,inner sep=2pt, label = right:{\scalebox{1.4}{$v_1$}}]  at (-1,3) {}; 
\node[fill=black,rectangle,inner sep=2pt, label = {[yshift=-0.2cm]right:{\scalebox{1.4}{$v_2$}}} ]  at (-1,2) {}; 
\node[fill=black,rectangle,inner sep=2pt, label = right:{\scalebox{1.4}{$v_3$}}]  at (-1,1) {}; 
\node[fill=black,rectangle,inner sep=2pt, label = {[yshift=0.2cm]right:{\scalebox{1.4}{$u_4$}}} ]  at (-1,0) {}; 
\node[fill=black,rectangle,inner sep=2pt, label = right:{\scalebox{1.4}{$v_5$}}]  at (-1,-1) {}; 
\node[fill=black,rectangle,inner sep=2pt, label = {[yshift=0.2cm]right:{\scalebox{1.4}{$v_6$}}} ]  at (-1,-2) {}; 
\node[fill=black,rectangle,inner sep=2pt, label = left:{\scalebox{1.4}{$u_1$}}] at (1,3) {}; 
\node[fill=black,rectangle,inner sep=2pt, label = left:{\scalebox{1.4}{$u_2$}}]  at (1,2) {}; 
\node[fill=black,rectangle,inner sep=2pt, label = left:{\scalebox{1.4}{$u_3$}}]  at (1,1) {}; 
\node[fill=black,rectangle,inner sep=2pt, label = {[yshift=0.2cm]left:{\scalebox{1.4}{$u_4$}}} ]  at (1,0) {}; 
\node[fill=black,rectangle,inner sep=2pt, label = left:{\scalebox{1.4}{$u_5$}}]  at (1,-1) {}; 
\node[fill=black,rectangle,inner sep=2pt, label = {[yshift=0.2cm]left:{\scalebox{1.4}{$u_6$}}} ]  at (1,-2) {}; 
\draw[dashed] (-1,3) -- (1,1);
\draw[dashed] (-1,2) -- (1,3);
\draw[dashed] (-1,1) -- (1,-1);
\draw[dashed] (-1,0) -- (1,0);
\draw[dashed] (-1,-1) -- (1,2);
\draw[dashed] (-1,-2) -- (1,-2);
\draw (-1,2) -- (-1.5,1.5); 
\draw (-1.5,-1.5) -- (-1,-1);
\draw (-1,0) -- (-2,-1);
\draw (-2,2) -- (-1,1); 
\draw (-1,3) -- (-3.5,0.5) -- (-1,-2);
\draw (1,2) --  (1.5,1.5)  -- (2,1) -- (3,0) -- (1.5,-1.5) -- (1,-1);
\draw (1,0) -- (1.5,0.5);
\draw (1,1) -- (2.5,-0.5);
\draw (1,3) -- (3.5,0.5) -- (1,-2);
\node[fill=black,circle,inner sep=1pt]  at (-2,2) {};
\node[fill=black,circle,inner sep=1pt]  at (-1.5,1.5) {};
\node[fill=black,circle,inner sep=1pt]  at (-1.5,-1.5) {};
\node[fill=black,circle,inner sep=1pt]  at (-2,-1) {};
\node[fill=black,circle,inner sep=1pt]  at (-3.5,0.5) {};
\node[fill=black,circle,inner sep=1pt]  at (1.5,-1.5) {};
\node[fill=black,circle,inner sep=1pt]  at (1.5,0.5) {};
\node[fill=black,circle,inner sep=1pt]  at (2.5,-0.5) {};
\node[fill=black,circle,inner sep=1pt]  at (3,0) {};
\node[fill=black,circle,inner sep=1pt]  at (3.5,0.5) {};
\node at (0,-3) {\huge{$\T_1$}};
}
}

\tikzset
{
example_6_2/.pic=
{%
\node[fill=black,rectangle,inner sep=2pt, label = right:{\scalebox{1.4}{$v_1$}}]  at (-1,3) {}; 
\node[fill=black,rectangle,inner sep=2pt, label = {[yshift=-0.2cm]right:{\scalebox{1.4}{$v_2$}}} ]  at (-1,2) {}; 
\node[fill=black,rectangle,inner sep=2pt, label = {[yshift=0.2cm]right:{\scalebox{1.4}{$u_3$}}} ]  at (-1,1) {}; 
\node[fill=black,rectangle,inner sep=2pt, label = {[yshift=0.2cm]right:{\scalebox{1.4}{$u_4$}}} ]  at (-1,0) {}; 
\node[fill=black,rectangle,inner sep=2pt, label = right:{\scalebox{1.4}{$v_5$}}]  at (-1,-1) {}; 
\node[fill=black,rectangle,inner sep=2pt, label = {[yshift=0.2cm]right:{\scalebox{1.4}{$v_6$}}} ]  at (-1,-2) {}; 
\node[fill=black,rectangle,inner sep=2pt, label = left:{\scalebox{1.4}{$u_1$}}] at (1,3) {}; 
\node[fill=black,rectangle,inner sep=2pt, label = left:{\scalebox{1.4}{$u_2$}}]  at (1,2) {}; 
\node[fill=black,rectangle,inner sep=2pt, label = {[yshift=0.2cm]left:{\scalebox{1.4}{$u_3$}}} ]  at (1,1) {}; 
\node[fill=black,rectangle,inner sep=2pt, label = {[yshift=0.2cm]left:{\scalebox{1.4}{$u_4$}}} ]  at (1,0) {}; 
\node[fill=black,rectangle,inner sep=2pt, label = left:{\scalebox{1.4}{$u_5$}}]  at (1,-1) {}; 
\node[fill=black,rectangle,inner sep=2pt, label = {[yshift=0.2cm]left:{\scalebox{1.4}{$u_6$}}} ]  at (1,-2) {}; 
\draw[dashed] (-1,3) -- (1,-1);
\draw[dashed] (-1,2) -- (1,3);
\draw[dashed] (-1,1) -- (1,1);
\draw[dashed] (-1,0) -- (1,0);
\draw[dashed] (-1,-1) -- (1,2);
\draw[dashed] (-1,-2) -- (1,-2);
\draw (-1,2) -- (-1.5,1.5); 
\draw (-1.5,-1.5) -- (-1,-1);
\draw (-1,0) -- (-2,-1);
\draw (-2,2) -- (-1,1); 
\draw (-1,3) -- (-3.5,0.5) -- (-1,-2);
\draw (1,2) -- (1.5,1.5)  -- (2,1) -- (3,0) -- (1.5,-1.5) -- (1,-1);
\draw (1,0) -- (1.5,0.5);
\draw (1,1) -- (2.5,-0.5);
\draw (1,3) -- (3.5,0.5) -- (1,-2);
\node[fill=black,circle,inner sep=1pt]  at (-2,2) {};
\node[fill=black,circle,inner sep=1pt]  at (-1.5,1.5) {};
\node[fill=black,circle,inner sep=1pt]  at (-1.5,-1.5) {};
\node[fill=black,circle,inner sep=1pt]  at (-2,-1) {};
\node[fill=black,circle,inner sep=1pt]  at (-3.5,0.5) {};
\node[fill=black,circle,inner sep=1pt]  at (1.5,-1.5) {};
\node[fill=black,circle,inner sep=1pt]  at (1.5,0.5) {};
\node[fill=black,circle,inner sep=1pt]  at (2.5,-0.5) {};
\node[fill=black,circle,inner sep=1pt]  at (3,0) {};
\node[fill=black,circle,inner sep=1pt]  at (3.5,0.5) {};
\node at (0,-3) {\huge{$\T_2$}};
}
}
\tikzset
{
example_6_card_1/.pic=
{%
\node[fill=black,rectangle,inner sep=2pt]  at (-1,2) {}; 
\node[fill=black,rectangle,inner sep=2pt]  at (-1,1) {}; 
\node[fill=black,rectangle,inner sep=2pt]  at (-1,0) {}; 
\node[fill=black,rectangle,inner sep=2pt]  at (-1,-1) {}; 
\node[fill=black,rectangle,inner sep=2pt]  at (-1,-2) {}; 
\draw[dashed] (-1,2) -- (1,1);
\draw[dashed] (-1,1) -- (1,-1);
\draw[dashed] (-1,0) -- (1,0);
\draw[dashed] (-1,-1) -- (1,2);
\draw[dashed] (-1,-2) -- (1,-2);
\draw (-1,2) -- (-1.5,1.5) -- (-2,1) -- (-3,0) -- (-1,-2) -- (-1.5,-1.5) -- (-1,-1);
\draw (-1,0) -- (-2,-1);
\draw (-1,1) -- (-1.4,1.5) ; 
\node[fill=black,rectangle,inner sep=2pt]  at (1,2) {}; 
\node[fill=black,rectangle,inner sep=2pt,]  at (1,1) {}; 
\node[fill=black,rectangle,inner sep=2pt]  at (1,0) {}; 
\node[fill=black,rectangle,inner sep=2pt]  at (1,-1) {}; 
\node[fill=black,rectangle,inner sep=2pt]  at (1,-2) {}; 
\draw (1,2)  -- (1.5,1.5) -- (2,1) -- (3,0) -- (1.5,-1.5) -- (1,-1);
\draw (1,0) -- (1.5,0.5);
\draw (1,1) -- (2.5,-0.5);
\draw (3,0) -- (1,-2);
\node[fill=black,circle,inner sep=1pt]  at (1.5,-1.5) {};
\node[fill=black,circle,inner sep=1pt]  at (1.5,0.5) {};
\node[fill=black,circle,inner sep=1pt]  at (2.5,-0.5) {};
\node[fill=black,circle,inner sep=1pt]  at (3,0) {};
\node at (0,-2.5) {$\mathcal{T}_1 - v_2u_1 = \mathcal{T}_1 - v_5u_2$};
\node at (0,-3) {$\mathcal{T}_2 - v_2u_1 = \mathcal{T}_2 - v_5u_2$};
}
}

\tikzset
{
example_6_card_2/.pic=
{%
\node[fill=black,rectangle,inner sep=2pt]  at (-1,2) {}; 
\node[fill=black,rectangle,inner sep=2pt]  at (-1,1) {}; 
\node[fill=black,rectangle,inner sep=2pt]  at (-1,0) {}; 
\node[fill=black,rectangle,inner sep=2pt]  at (-1,-1) {}; 
\node[fill=black,rectangle,inner sep=2pt]  at (-1,-2) {}; 
\draw[dashed] (-1,2) -- (1,2);
\draw[dashed] (-1,1) -- (1,-1);
\draw[dashed] (-1,0) -- (1,0);
\draw[dashed] (-1,-1) -- (1,1);
\draw[dashed] (-1,-2) -- (1,-2);
\draw (-1,2) -- (-1.5,1.5)  -- (-2,1) -- (-3,0) -- (-1,-2) -- (-1.5,-1.5) -- (-1,-1);
\draw (-1,0) -- (-2,-1);
\draw (-1,1) -- (-1.5,1.5) ; 
\node[fill=black,circle,inner sep=1pt]  at (-1.5,1.5) {};
\node[fill=black,circle,inner sep=1pt]  at (-1.5,-1.5) {};
\node[fill=black,circle,inner sep=1pt]  at (-2,-1) {};
\node[fill=black,circle,inner sep=1pt]  at (-3,0) {};
\node[fill=black,rectangle,inner sep=2pt]  at (1,2) {}; 
\node[fill=black,rectangle,inner sep=2pt]  at (1,1) {}; 
\node[fill=black,rectangle,inner sep=2pt]  at (1,0) {}; 
\node[fill=black,rectangle,inner sep=2pt]  at (1,-1) {}; 
\node[fill=black,rectangle,inner sep=2pt]  at (1,-2) {}; 
\draw (1,2) -- (1.5,1.5)  -- (2,1) -- (3,0) -- (1,-2) -- (1.5,-1.5) -- (1,-1);
\draw (1,0) -- (2,-1);
\draw (1,1) -- (2.5,-0.5);
\node[fill=black,circle,inner sep=1pt]  at (1.5,-1.5) {};
\node[fill=black,circle,inner sep=1pt]  at (2,-1) {};
\node[fill=black,circle,inner sep=1pt]  at (2.5,-0.5) {};
\node[fill=black,circle,inner sep=1pt]  at (3,0) {};
\node at (0,-2.5) {$\mathcal{T}_1 - v_1u_3$};
\node at (0,-3) {$\T_2 - v_3u_3$};
}
}

\tikzset
{
example_6_card_3/.pic=
{%
\node[fill=black,rectangle,inner sep=2pt]  at (-1,2) {}; 
\node[fill=black,rectangle,inner sep=2pt]  at (-1,1) {}; 
\node[fill=black,rectangle,inner sep=2pt]  at (-1,0) {}; 
\node[fill=black,rectangle,inner sep=2pt]  at (-1,-1) {}; 
\node[fill=black,rectangle,inner sep=2pt]  at (-1,-2) {}; 
\draw[dashed] (-1,2) -- (1,2);
\draw[dashed] (-1,1) -- (1,-1);
\draw[dashed] (-1,0) -- (1,-2);
\draw[dashed] (-1,-1) -- (1,1);
\draw[dashed] (-1,-2) -- (1,0);
\draw (-1,2)  -- (-1.5,1.5) -- (-2,1) -- (-3,0) -- (-1,-2) -- (-1.5,-1.5) -- (-1,-1);
\draw (-1,0) -- (-2,-1);
\draw (-1,1) -- (-1.5,1.5) ; 
\node[fill=black,circle,inner sep=1pt]  at (-1.5,1.5) {};
\node[fill=black,circle,inner sep=1pt]  at (-1.5,-1.5) {};
\node[fill=black,circle,inner sep=1pt]  at (-2,-1) {};
\node[fill=black,circle,inner sep=1pt]  at (-3,0) {};
\node[fill=black,rectangle,inner sep=2pt]  at (1,2) {}; 
\node[fill=black,rectangle,inner sep=2pt]  at (1,1) {}; 
\node[fill=black,rectangle,inner sep=2pt]  at (1,0) {}; 
\node[fill=black,rectangle,inner sep=2pt]  at (1,-1) {}; 
\node[fill=black,rectangle,inner sep=2pt]  at (1,-2) {}; 
\draw (1,2) -- (1.5,1.5)  -- (2,1) -- (3,0) -- (1,-2) -- (1.5,-1.5) -- (1,-1);
\draw (1,0) -- (2,-1);
\draw (1,1) -- (2.5,-0.5);
\node[fill=black,circle,inner sep=1pt]  at (1.5,-1.5) {};
\node[fill=black,circle,inner sep=1pt]  at (2,-1) {};
\node[fill=black,circle,inner sep=1pt]  at (2.5,-0.5) {};
\node[fill=black,circle,inner sep=1pt]  at (3,0) {};
\node at (0,-2.5) {$\mathcal{T}_1 - v_3u_5$};
\node at (0,-3) {$\T_2 - v_1u_5$};
}
}

\tikzset
{
example_6_card_4/.pic=
{%
\node[fill=black,rectangle,inner sep=2pt]  at (-1,2) {}; 
\node[fill=black,rectangle,inner sep=2pt]  at (-1,1) {}; 
\node[fill=black,rectangle,inner sep=2pt]  at (-1,0) {}; 
\node[fill=black,rectangle,inner sep=2pt]  at (-1,-1) {}; 
\node[fill=black,rectangle,inner sep=2pt]  at (-1,-2) {}; 
\draw[dashed] (-1,2) -- (1,1);
\draw[dashed] (-1,1) -- (1,-1);
\draw[dashed] (-1,0) -- (1,0);
\draw[dashed] (-1,-1) -- (1,2);
\draw[dashed] (-1,-2) -- (1,-2);
\draw (-1,2) --  (-1.5,1.5) -- (-2,1) -- (-3,0) -- (-1,-2) -- (-1.5,-1.5) -- (-1,-1);
\draw (-1,0) -- (-2,-1);
\draw (-1,1) -- (-1.5,1.5) ; 
\node[fill=black,circle,inner sep=1pt]  at (-1.5,1.5) {};
\node[fill=black,circle,inner sep=1pt]  at (-1.5,-1.5) {};
\node[fill=black,circle,inner sep=1pt]  at (-2,-1) {};
\node[fill=black,circle,inner sep=1pt]  at (-3,0) {};
\node[fill=black,rectangle,inner sep=2pt]  at (1,2) {}; 
\node[fill=black,rectangle,inner sep=2pt]  at (1,1) {}; 
\node[fill=black,rectangle,inner sep=2pt]  at (1,0) {}; 
\node[fill=black,rectangle,inner sep=2pt]  at (1,-1) {}; 
\node[fill=black,rectangle,inner sep=2pt]  at (1,-2) {}; 
\draw (1,2)  -- (1.5,1.5) -- (2,1) -- (3,0) -- (1,-2) -- (1.5,-1.5) -- (1,-1);
\draw (1,0) -- (2,-1);
\draw (1,1) -- (2.5,-0.5);
\node[fill=black,circle,inner sep=1pt]  at (1.5,-1.5) {};
\node[fill=black,circle,inner sep=1pt]  at (2,-1) {};
\node[fill=black,circle,inner sep=1pt]  at (2.5,-0.5) {};
\node[fill=black,circle,inner sep=1pt]  at (3,0) {};
\node at (0,-2.5) {$\mathcal{T}_1 - v_4u_4$};
\node at (0,-3) {$\T_2 - v_6u_6$};
}
}

\tikzset
{
example_6_card_5/.pic=
{%
\node[fill=black,rectangle,inner sep=2pt]  at (-1,2) {}; 
\node[fill=black,rectangle,inner sep=2pt]  at (-1,1) {}; 
\node[fill=black,rectangle,inner sep=2pt]  at (-1,0) {}; 
\node[fill=black,rectangle,inner sep=2pt]  at (-1,-1) {}; 
\node[fill=black,rectangle,inner sep=2pt]  at (-1,-2) {}; 
\draw[dashed] (-1,2) -- (1,1);
\draw[dashed] (-1,1) -- (1,-1);
\draw[dashed] (-1,0) -- (1,-2);
\draw[dashed] (-1,-1) -- (1,2);
\draw[dashed] (-1,-2) -- (1,0);
\draw (-1,2)  -- (-1.5,1.5)  -- (-2,1) -- (-3,0) -- (-1,-2) -- (-1.5,-1.5) -- (-1,-1);
\draw (-1,0) -- (-2,-1);
\draw (-1,1) -- (-1.5,1.5) ; 
\node[fill=black,circle,inner sep=1pt]  at (-1.5,1.5) {};
\node[fill=black,circle,inner sep=1pt]  at (-1.5,-1.5) {};
\node[fill=black,circle,inner sep=1pt]  at (-2,-1) {};
\node[fill=black,circle,inner sep=1pt]  at (-3,0) {};
\node[fill=black,rectangle,inner sep=2pt]  at (1,2) {}; 
\node[fill=black,rectangle,inner sep=2pt]  at (1,1) {}; 
\node[fill=black,rectangle,inner sep=2pt]  at (1,0) {}; 
\node[fill=black,rectangle,inner sep=2pt]  at (1,-1) {}; 
\node[fill=black,rectangle,inner sep=2pt]  at (1,-2) {}; 
\draw (1,2)  -- (1.5,1.5)  -- (2,1) -- (3,0) -- (1,-2) -- (1.5,-1.5) -- (1,-1);
\draw (1,0) -- (2,-1);
\draw (1,1) -- (2.5,-0.5);
\node[fill=black,circle,inner sep=1pt]  at (1.5,-1.5) {};
\node[fill=black,circle,inner sep=1pt]  at (2,-1) {};
\node[fill=black,circle,inner sep=1pt]  at (2.5,-0.5) {};
\node[fill=black,circle,inner sep=1pt]  at (3,0) {};
\node at (0,-2.5) {$\mathcal{T}_1 - v_6u_6$};
\node at (0,-3) {$\T_2 - v_4u_4$};
}
}

\title[Reconstruction of Caterpillar Tanglegrams]{Reconstruction of Caterpillar Tanglegrams}
 \author[A. Clifton]{Ann Clifton}
 \author[\'E. Czabarka]{\'Eva Czabarka}
 \author[K. Liu]{Kevin Liu} 
 \author[S. Loeb]{Sarah Loeb}
\author[U. Okur]{Utku Okur}
\author[L. Sz\'ekely]{L\'aszl\'o Sz\'ekely}
\author[K. Wicke]{Kristina Wicke}
 \address{Ann Clifton\\ Louisiana Tech University}
 \email{aclifton@latech.edu}
 \address{\'Eva Czabarka\\ University of South Carolina}
 \email{czabarka@math.sc.edu}
 \address{Kevin Liu\\The University of the South}
 \email{keliu@sewanee.edu}
 \address{Sarah Loeb\\ Hampden-Sydney College}
 \email{sloeb@hsc.edu}
\address{Utku Okur\\ University of South Carolina}
\email{uokur@email.sc.edu}
\address{L\'aszl\'o Sz\'ekely\\ University of South Carolina}
\email{szekely@math.sc.edu}
\address{Kristina Wicke\\  New Jersey Institute of Technology}
\email{kristina.wicke@njit.edu}

\begin{document}

\begin{abstract}
A tanglegram consists of two rooted binary trees with the same number of leaves and a perfect matching between the leaves of the trees. Given a size-$n$ tanglegram, i.e., a tanglegram for two trees with $n$ leaves, a multiset of induced size-$(n-1)$ tanglegrams is obtained by deleting a pair of matched leaves in every possible way. Here, we analyze whether a size-$n$ tanglegram is uniquely encoded by this multiset of size-$(n-1)$ tanglegrams. We answer this question affirmatively in the case that at least one of the two trees of the tanglegram is a caterpillar tree. 
\end{abstract}
\maketitle


\section{Introduction}
A \emph{tanglegram} $\T$ consists of two rooted binary trees $L$ and $R$ with the same number of leaves and a perfect matching $\sigma$ between the two leaf sets. In biology, tanglegrams are used in the study of cospeciation and coevolution. For example, tree $L$ may represent the phylogeny of a host, tree $R$ the phylogeny of a parasite, and the matching $\sigma$ the connection between the host and the parasite \cite{Matsen2015,Venkatachalam2009}. Tanglegrams lead to a variety of mathematical questions, including their enumeration \cite{Billey2017,Gessel2021,Wagner2018}, drawing them with the fewest number of crossings possible \cite{buchin2012,Fernau2005,Venkatachalam2009}, and their typical shape~\cite{Konvalinka2016}. 

In this paper, we study the \emph{tanglegram reconstruction problem} originally posed by  Stephan Wagner and L\'aszl\'o Sz\'ekely in 2017: Given a tanglegram  $\T$ of size $n$ (where $n$ is the number of leaves in $L$ or $R$), does the multiset $\D^m(\T)$ of induced size-$(n-1)$ tanglegrams obtained from $\T$ by deleting a pair of matched leaves from $\T$ in every possible way uniquely determine $\T$ for $n$ sufficiently large? This problem is motivated by the famous graph reconstruction conjecture that states that every graph $G$ on at least three vertices is uniquely determined (up to isomorphism) from the multiset of graphs obtained by deleting one vertex in every possible way from $G$~\cite{Kelly1942, Ulam1960}. This conjecture has been proven for many special classes of graphs but remains open in general (see \cite{Bondy1977, Harary1974, Ramachandran2004} for an overview). It has also motivated various other reconstruction problems such as for partially ordered sets~\cite{Kratsch1994}, permutations~\cite{Smith2006}, and groups~\cite{Radcliffe2006}. 

Here, we affirmatively answer the tanglegram reconstruction problem for the case that at least one of the two trees of the tanglegram is a caterpillar tree, where a \emph{caterpillar} tree is a tree whose internal vertices form a path, such that the root of the tree is an endpoint of the path. We establish the following main theorem.

\begin{theorem}\label{mainthm1}
    Let $\mathcal{T}=(L,R,\sigma)$ be a tanglegram of size $n\geq 6$. If $L$ or $R$ is a caterpillar, then $T$ is reconstructable from the multideck $\Dm(\T)$.
\end{theorem}

The remainder of the paper is organized as follows. In Section~\ref{Sec:Preliminaries}, we introduce general terminology and notation. We then define different types of trees in Section~\ref{sec:types}. In Section~\ref{sec:small} we discuss tanglegrams of size at most $5$ -- in particular we explain that Theorem~\ref{mainthm1} is best possible in the sense that the lower bound $n\ge 6$ can not be changed to $n\ge 5$. 
In Sections~\ref{Sec:CatergramReconstruction}--\ref{Sec:Type2Reconstruction}, we consider the tanglegram reconstruction problem involving caterpillar trees, distinguishing three cases, depending on the type of the tree the caterpillar is matched with.
 In Section~\ref{sec:conclusion} we finally prove Theorem~\ref{mainthm1} and pose a remaining problem.

\section{Preliminaries}\label{Sec:Preliminaries}
We begin by introducing terminology and notation.

\begin{definition}\label{def:tree}
A \emph{rooted tree} is a tree $T$ where a special vertex, denoted $r_T$, is identified as the root of $T$.
The \emph{parent} of a non-root vertex $v$ is $v$'s neighbor on the $r_T-v$ path in $T$ and $y$ is a \emph{child} of $v$ if $v$ is the parent of $y$.
The \emph{leaves} of $T$ are the vertices with no children. Two leaves with a common parent are called a \emph{cherry}. We let $\mathbb{L}(T)$ denote the set of leaves and $\size{T}$ denote the number of leaves of $T$. We refer to $\size{T}$ as the \emph{size} of $T$.
\end{definition} 

\begin{definition}\label{def:binary-tree}
    A \emph{rooted binary tree} is a rooted tree in which every vertex has zero or two children.  
\end{definition}
   
We often simply refer to this as a \emph{tree}, as we will not consider other types of trees. Note that in a rooted binary tree we have that either the root is a leaf (and the tree has size 1) or the root has degree 2 and non-root internal vertices have degree 3.

We remark that whenever we write $T = T'$, we mean that $T$ and $T'$ are equal up to an isomorphism that maps root to root. 

\begin{definition}\label{def:treedrawing}
    If $T$ is a rooted binary tree, a \emph{standard drawing of $T$} is a plane drawing of $T$ such that its leaves lie on a straight line (called the \emph{leaf-line of $T$}), the tree is drawn with straight lines on one of the two half-planes the line defines in the plane, and if $x,y$ are vertices of $T$ such that $x$ is the parent of $y$, then the distance of $x$ from the leaf-line is larger than the distance of $y$ from the leaf-line.  
\end{definition}

In general, a rooted binary tree has multiple standard drawings, as shown in Figure~\ref{fig:treeexamples}.

\begin{figure}[htbp]
\centering 
\scalebox{0.75}{
    \begin{tikzpicture}
        \pic at (-6,0) {tree5.1}; 
        \draw (0,0) -- (2,2) -- (2.5,1.5) -- (1,0);
        \draw (2.5,1.5) -- (3,1) -- (2,0);
        \draw (2.5,0.5) -- (3,0);
        \draw (3,1) -- (4,0);
        \node[fill=black,rectangle,inner sep=2pt]  at (0,0) {};      
        \node[fill=black,rectangle,inner sep=2pt]  at (1,0) {};      
        \node[fill=black,rectangle,inner sep=2pt]  at (2,0) {}; 
        \node[fill=black,rectangle,inner sep=2pt]  at (3,0) {}; 
        \node[fill=black,rectangle,inner sep=2pt]  at (4,0) {}; 
        \node[fill=black,circle,inner sep=1pt]  at (2,2) {};
        \node[fill=black,circle,inner sep=1pt]  at (2.5,1.5) {};
        \node[fill=black,circle,inner sep=1pt]  at (3,1) {};
        \node[fill=black,circle,inner sep=1pt]  at (2.5,0.5) {};
    \end{tikzpicture}
    }
    \caption{Two standard drawings for a rooted binary tree.}
    \label{fig:treeexamples}
\end{figure}

\begin{definition}\label{def:decomposition}
    Let $T_1, T_2$ be rooted binary trees and let $T_1 \oplus T_2$ be the rooted binary tree $T$ obtained by taking a new vertex $r_T$ and joining the roots of $T_1$ and $T_2$ to $r_T$. If $T = T_1 \oplus T_2$, we say \emph{$T$ is composed of $T_1$ and $T_2$}. Furthermore, we call $T_1$ and $T_2$ the \emph{pending subtrees} of $T$.
\end{definition}

The operation $\oplus$ is commutative. However, if not stated otherwise, for $T = T_1 \oplus T_2$, we assume $\size{T_1}  \leq \size{T_2}$.

\begin{definition}\label{def:caterpillar}
    The \emph{caterpillar} on $n$ leaves, denoted $C_n$, is the rooted binary tree defined recursively as follows:  $C_1$ is the unique rooted binary tree of size 1, and for $n>1$, $C_n=C_1 \oplus C_{n-1}$.
\end{definition}

Alternatively, one can view the caterpillar of size $n>1$ as the unique tree in which the removal of all leaves results in the path graph on $n-1$ vertices, such that the root of the tree is an endvertex of the path. Note that the tree in Fig.~\ref{fig:treeexamples} is $C_5$.

We next turn to induced subtrees and decks. We first recall an operation on graphs. Let $G$ be a graph and let $w$ be a degree-2 vertex with neighbors $u$ and $v$. Then, \emph{suppressing} $w$ means deleting $w$ and its two incident edges, $uw$ and $vw$, and introducing the edge $uv$. Note that this is the same as contracting either of the edges $uw$ or $vw$.

\begin{definition}\label{def:inducedsubtree}
Let $T$ be a rooted binary tree with leaf set $\mathbb{L}(T)$ and let $S\subseteq \mathbb{L}(T)$. The \emph{rooted binary tree induced by $S$}, denoted $T[S]$, is the tree formed by starting with the minimal subtree of $T$ containing all the vertices in $S$, declaring as root its closest vertex to the root of $T$, and then suppressing all non-root degree-2 vertices. We also refer to $T[S]$ as an \emph{induced subtree} of $T$.
\end{definition}

\begin{definition}\label{def:deck} 
The \emph{deck} $\mathcal{D}(T)$ of $T$ is the set of trees $T'$ with $\size{T'}=\size{T}-1$ such that $T'$ is an induced subtree of $T$. The \emph{multideck $\Dm(T)$} is a multiset with the same elements as the deck, but they come with multiplicities. To determine these multiplicities, label the leaves of $T$ arbitrarily, and denote the set of labeled leaves by $\mathbb{L}^*(T)$.
The multiplicity of an element $T'\in\mathcal{D}(T)$ in $\Dm(T)$ is $\left\vert\{S\subseteq\mathbb{L}^*(T): T'=T[S]\}\right\vert$.
\end{definition}

We often refer to the elements of $\D(T)$ (respectively $\D^m(T)$) as \emph{cards}.

\begin{definition}\label{def:reconstructable}
A 
tree $T$ is \emph{reconstructable} from its deck (resp. multideck) if 
for any tree $T'$, $\D(T')=\D(T)$ (resp. $\Dm(T')=\Dm(T)$) implies $T=T'$.
\end{definition}

The following result shows that for sufficiently large trees, trees are reconstructable from decks and multidecks.

\begin{theorem}[\cite{Clifton2025}]\label{thm:tree_reconstruction}
    For any tree of size $n\geq 5$, $T$ is reconstructable from the multideck $\Dm(T)$. If $n\geq 6$, then $T$ is reconstructable from the deck $\D(T)$. 
\end{theorem}

We are now in the position to define the central concept of this paper.
\begin{definition}\label{def:tanglegram}
    A \emph{tanglegram} $\T=(L,R,\sigma)$ consists of two rooted binary trees, the \emph{left tree} $L$ and the \emph{right tree} $R$, along with a perfect matching $\sigma$ between their leaves. We will denote elements of $\sigma$ using edges $vu$, where $v\in \mathbb{L}(L)$ and $u\in \mathbb{L}(R)$. For an edge $vu \in \sigma$, we also say that $v$ and $u$ are \emph{matched}. 
The \emph{size} $\size{\T}$ of a tanglegram $\T$ is the common size of its underlying trees, $|\sigma|$. 
\end{definition}

An isomorphism between two tanglegrams is a graph isomorphism that maps the root of the left (right) tree to the root of the right (left) tree. 

\begin{definition}
A \emph{layout of a tanglegram $\T=(L,R,\sigma)$}  is a straight line drawing such that for some $a<b$ we have a standard drawing of $L$ with a vertical leaf-line $x=a$  to the left of the leaf-line, a standard drawing of $R$ with a vertical leaf-line $x=b$ to the right of the leaf-line, and the edges in $\sigma$ are drawn as straight lines.   
\end{definition}

Tanglegrams can have several layouts, as shown in Figure~\ref{fig:tanglegramexamples}.

\begin{figure}[htbp]
\scalebox{0.6}{
\begin{tikzpicture}
\node[fill=black,rectangle,inner sep=2pt]  at (-1,2) {}; 
\node[fill=black,rectangle,inner sep=2pt]  at (-1,1) {}; 
\node[fill=black,rectangle,inner sep=2pt]  at (-1,0) {}; 
\node[fill=black,rectangle,inner sep=2pt]  at (-1,-1) {}; 
\node[fill=black,rectangle,inner sep=2pt]  at (-1,-2) {}; 
\node[fill=black,rectangle,inner sep=2pt]  at (1,2) {}; 
\node[fill=black,rectangle,inner sep=2pt]  at (1,1) {}; 
\node[fill=black,rectangle,inner sep=2pt]  at (1,0) {}; 
\node[fill=black,rectangle,inner sep=2pt]  at (1,-1) {}; 
\node[fill=black,rectangle,inner sep=2pt]  at (1,-2) {}; 
\draw[dashed] (-1,2) -- (1,-1);
\draw[dashed] (-1,1) -- (1,1);
\draw[dashed] (-1,0) -- (1,-2);
\draw[dashed] (-1,-1) -- (1,2);
\draw[dashed] (-1,-2) -- (1,0);
\draw (-1,2) -- (-1.5,1.5) -- (-1,1) -- (-1.5,1.5) -- (-2,1) -- (-1,0) -- (-2,1) -- (-3,0) -- (-1,-2) -- (-1.5,-1.5) -- (-1,-1);
\draw (1,2) -- (2,1) --(1.5,0.5) -- (1,1) -- (1.5,0.5) -- (1,0) -- (1.5,0.5) -- (2,1) -- (2.5,.5) -- (1,-1) -- (2.5,.5)-- (3,0) -- (1,-2);
\node[fill=black,circle,inner sep=1pt]  at (-1.5,1.5) {};
\node[fill=black,circle,inner sep=1pt]  at (-2,1) {};
\node[fill=black,circle,inner sep=1pt]  at (-1.5,-1.5) {};
\node[fill=black,circle,inner sep=1pt]  at (-3,0) {};
\node[fill=black,circle,inner sep=1pt]  at (2,1) {};
\node[fill=black,circle,inner sep=1pt]  at (1.5,0.5) {};
\node[fill=black,circle,inner sep=1pt]  at (2.5,0.5) {};
\node[fill=black,circle,inner sep=1pt]  at (3,0) {};
\end{tikzpicture} \qquad
\begin{tikzpicture}
\node[fill=black,rectangle,inner sep=2pt]  at (-1,2) {}; 
\node[fill=black,rectangle,inner sep=2pt]  at (-1,1) {}; 
\node[fill=black,rectangle,inner sep=2pt]  at (-1,0) {}; 
\node[fill=black,rectangle,inner sep=2pt]  at (-1,-1) {}; 
\node[fill=black,rectangle,inner sep=2pt]  at (-1,-2) {}; 
\node[fill=black,rectangle,inner sep=2pt]  at (1,2) {}; 
\node[fill=black,rectangle,inner sep=2pt]  at (1,1) {}; 
\node[fill=black,rectangle,inner sep=2pt]  at (1,0) {}; 
\node[fill=black,rectangle,inner sep=2pt]  at (1,-1) {}; 
\node[fill=black,rectangle,inner sep=2pt]  at (1,-2) {};
\draw[dashed] (-1,2) -- (1,2);
\draw[dashed] (-1,1) -- (1,1);
\draw[dashed] (-1,0) -- (1,0);
\draw[dashed] (-1,-1) -- (1,-1);
\draw[dashed] (-1,-2) -- (1,-2);
\draw (-1,0) -- (-1.5,0.5) -- (-1,1) -- (-1.5,0.5) -- (-2,1) -- (-1,2) -- (-2,1) -- (-3,0) -- (-1,-2) -- (-1.5,-1.5) -- (-1,-1);
\draw (1,1) -- (2.5,-0.5) -- (2,-1) -- (1,0) -- (1.5,-0.5) -- (1,-1) -- (1.5,-0.5) -- (2,-1) -- (1,-2) -- (3,0) -- (1,2);
\node[fill=black,circle,inner sep=1pt]  at (-1.5,0.5) {};
\node[fill=black,circle,inner sep=1pt]  at (-2,1) {};
\node[fill=black,circle,inner sep=1pt]  at (-1.5,-1.5) {};
\node[fill=black,circle,inner sep=1pt]  at (-3,0) {};
\node[fill=black,circle,inner sep=1pt]  at (2.5,-0.5) {};
\node[fill=black,circle,inner sep=1pt]  at (2,-1) {};
\node[fill=black,circle,inner sep=1pt]  at (1.5,-0.5) {};
\node[fill=black,circle,inner sep=1pt]  at (3,0) {};
\end{tikzpicture}
}
\caption{Two layouts of the same tanglegram.}
\label{fig:tanglegramexamples}
\end{figure}

We next define induced subtanglegrams of a tanglegram $\T=(L,R,\sigma)$, as well as the (multi)deck of a tanglegram.

\begin{definition}\label{def:induced-subtanglegram}
Let $\T=(L,R,\sigma)$ be a tanglegram. For any  $E\subseteq \sigma$, let $S_L\subseteq \mathbb{L}(L)$ and 
$S_R\subseteq \mathbb{L}(R)$ be the set of leaves matched by $E$. The \emph{subtanglegram} $\T[E]$ \emph{induced} by $E$ is defined as the tanglegram formed from the induced subtrees $L[S_L]$ and $R[S_R]$ with the matching given by $E$. 
\end{definition}

When $E=\sigma-\{vu\}$ for some $vu\in\sigma$, we abuse notation and also denote $\T[\sigma-\{vu\}]$ as $\T-vu$. If the focus is on a particular leaf removed from $L$ (respectively $R$), then we abuse notation and also denote this as as $\T-v$ (respectively $\T-u$).

\begin{definition}\label{def:deck-tanglegrams} 
The \emph{deck} $\D(\T)$ of the tanglegram $\T$ is the set of size $\size{\T}-1$ tanglegrams that are induced subtanglegrams of $\T$. 
The \emph{multideck} $\D^{m}(\T)$ of the tanglegram $\T$ is a multiset whose elements are the elements of $\D(\T)$. To 
determine these multiplicities, label the edges of $\sigma$ arbitrarily, and denote the labeled set by $\sigma^*$. 
The multiplicity of an element $\T'\in\D(\T)$ in $\Dm(\T)$ is 
$\vert\{E\subseteq\sigma^*: \T'=\T[E]\}\vert$.
\end{definition}
As for trees, we often refer to the elements of $\D^m(\T)$ (respectively $\D(\T)$) as \emph{cards}. 
An example is shown in Fig.~\ref{fig:tanglegramdeck}.

\begin{figure}[htbp]
    \centering
    \scalebox{0.6}{
    \begin{tikzpicture}
\node[fill=black,rectangle,inner sep=2pt]  at (-1,2) {}; 
\node[fill=black,rectangle,inner sep=2pt]  at (-1,1) {}; 
\node[fill=black,rectangle,inner sep=2pt]  at (-1,0) {}; 
\node[fill=black,rectangle,inner sep=2pt]  at (-1,-1) {}; 
\node[fill=black,rectangle,inner sep=2pt]  at (-1,-2) {}; 
\node[fill=black,rectangle,inner sep=2pt]  at (1,2) {}; 
\node[fill=black,rectangle,inner sep=2pt]  at (1,1) {}; 
\node[fill=black,rectangle,inner sep=2pt]  at (1,0) {}; 
\node[fill=black,rectangle,inner sep=2pt]  at (1,-1) {}; 
\node[fill=black,rectangle,inner sep=2pt]  at (1,-2) {};
\draw[dashed] (-1,2) -- (1,2);
\draw[dashed] (-1,1) -- (1,1);
\draw[dashed] (-1,0) -- (1,0);
\draw[dashed] (-1,-1) -- (1,-1);
\draw[dashed] (-1,-2) -- (1,-2);
\node at (0,2.2) {$e_1$};
\node at (0,1.2) {$e_2$};
\node at (0,.2) {$e_3$};
\node at (0,-.8) {$e_4$};
\node at (0,-1.8) {$e_5$};
\draw (-1,0) -- (-1.5,0.5) -- (-1,1) -- (-1.5,0.5) -- (-2,1) -- (-1,2) -- (-2,1) -- (-3,0) -- (-1,-2) -- (-1.5,-1.5) -- (-1,-1);
\draw (1,1) -- (2.5,-0.5) -- (2,-1) -- (1,0) -- (1.5,-0.5) -- (1,-1) -- (1.5,-0.5) -- (2,-1) -- (1,-2) -- (3,0) -- (1,2);
\node at (0,-2.5) {\huge{$\mathcal{T}$}};
\end{tikzpicture}}\\
\scalebox{0.6}{
\begin{tikzpicture}
    \node at (0,4) {\phantom{-}};
    \node at (0,0) {\phantom{-}};
    \draw (-1,0) -- (-1.5,0.5) -- (-1,1) -- (-1.5,0.5) -- (-2.5,1.5) -- (-1,3) -- (-1.5,2.5) -- (-1,2);
    \draw (1,0) -- (2,1) -- (1,2) -- (1.5,1.5) -- (1,1) -- (1.5,1.5) -- (2,1) -- (2.5,1.5) -- (1,3);
    \node[fill=black,rectangle,inner sep=2pt]  at (-1,2) {}; 
    \node[fill=black,rectangle,inner sep=2pt]  at (-1,1) {}; 
    \node[fill=black,rectangle,inner sep=2pt]  at (-1,0) {}; 
    \node[fill=black,rectangle,inner sep=2pt]  at (-1,3) {}; 
    \node[fill=black,rectangle,inner sep=2pt]  at (1,3) {}; 
    \node[fill=black,rectangle,inner sep=2pt]  at (1,2) {}; 
    \node[fill=black,rectangle,inner sep=2pt]  at (1,1) {}; 
    \node[fill=black,rectangle,inner sep=2pt]  at (1,0) {}; 
    \draw[dashed] (-1,0) -- (1,0);
    \draw[dashed] (-1,1) -- (1,1);
    \draw[dashed] (-1,2) -- (1,2);
    \draw[dashed] (-1,3) -- (1,3);
    \node at (0,-0.75) {\huge{$\mathcal{T}_1=\T-e_1=\T-e_2$}};
\end{tikzpicture}
\qquad 
\begin{tikzpicture}
    \node at (0,4) {\phantom{-}};
    \node at (0,0) {\phantom{-}};
    \draw (-1,0) -- (-1.5,0.5) -- (-1,1) -- (-1.5,0.5) -- (-2.5,1.5) -- (-1,3) -- (-1.5,2.5) -- (-1,2);
    \draw (1,0) -- (1.5,0.5) -- (1,1) -- (1.5,0.5) -- (2,1) -- (1,2) -- (2,1) -- (2.5,1.5) -- (1,3);
    \node[fill=black,rectangle,inner sep=2pt]  at (-1,2) {}; 
    \node[fill=black,rectangle,inner sep=2pt]  at (-1,1) {}; 
    \node[fill=black,rectangle,inner sep=2pt]  at (-1,0) {}; 
    \node[fill=black,rectangle,inner sep=2pt]  at (-1,3) {}; 
    \node[fill=black,rectangle,inner sep=2pt]  at (1,3) {}; 
    \node[fill=black,rectangle,inner sep=2pt]  at (1,2) {}; 
    \node[fill=black,rectangle,inner sep=2pt]  at (1,1) {}; 
    \node[fill=black,rectangle,inner sep=2pt]  at (1,0) {}; 
    \draw[dashed] (-1,0) -- (1,0);
    \draw[dashed] (-1,1) -- (1,1);
    \draw[dashed] (-1,2) -- (1,2);
    \draw[dashed] (-1,3) -- (1,3);
    \node at (0,-0.75) {\huge{$\mathcal{T}_2=\T-e_3$}};
\end{tikzpicture}
\qquad 
\begin{tikzpicture}
    \node at (0,4) {\phantom{-}};
    \node at (0,0) {\phantom{-}};
    \draw (-1,0) -- (-2.5,1.5) -- (-2,2) -- (-1,1) -- (-1.5,1.5) -- (-1,2) -- (-1.5,1.5) -- (-2,2) -- (-1,3);
    \draw (1,0) -- (1.5,0.5) -- (1,1) -- (1.5,0.5) -- (2,1) -- (1,2) -- (2,1) -- (2.5,1.5) -- (1,3);
    \node[fill=black,rectangle,inner sep=2pt]  at (-1,2) {}; 
    \node[fill=black,rectangle,inner sep=2pt]  at (-1,1) {}; 
    \node[fill=black,rectangle,inner sep=2pt]  at (-1,0) {}; 
    \node[fill=black,rectangle,inner sep=2pt]  at (-1,3) {}; 
    \node[fill=black,rectangle,inner sep=2pt]  at (1,3) {}; 
    \node[fill=black,rectangle,inner sep=2pt]  at (1,2) {}; 
    \node[fill=black,rectangle,inner sep=2pt]  at (1,1) {}; 
    \node[fill=black,rectangle,inner sep=2pt]  at (1,0) {}; 
    \draw[dashed] (-1,0) -- (1,0);
    \draw[dashed] (-1,1) -- (1,1);
    \draw[dashed] (-1,2) -- (1,2);
    \draw[dashed] (-1,3) -- (1,3);
    \node at (0,-0.75) {\huge{$\mathcal{T}_3=\T-e_4=\T-e_5$}};
\end{tikzpicture}
}
    \caption{The deck of $\mathcal{T}$ is $\{\mathcal{T}_1,\mathcal{T}_2,\mathcal{T}_3\}$. The multideck of $\T$ contains two copies of $\T_1$, one copy of $\T_2$, and two copies of $\T_3$.}
    \label{fig:tanglegramdeck}
\end{figure}

\begin{definition}\label{def:tanglereconstructable}
A 
tanglegran $\T$ is \emph{reconstructable} from its deck (resp. multideck) if 
for any tanglegram $\T'$, $\D(\T')=\D(\T)$ (resp. $\Dm(\T')=\Dm(\T)$) implies $\T=\T'$.
\end{definition}

We are interested in determining whether all tanglegrams or all tanglegrams of a certain class can be reconstructed from their deck. This brings up the following problem: Say, we have a tanglegram class $\mathcal{H}$ (e.g. $\mathcal{H}$ is the class of tanglegrams with at least one of the left or right tree being a caterpillar) and we know that all tanglegrams in $\mathcal{H}$ have unique decks, but there are some decks that also correspond to tanglegrams not in $\mathcal{H}$. In this case, given a deck of a tanglegram, we still have no information on whether we can reconstruct the tanglegram. Because of this, we require in our definition that we can determine from the deck (or multideck) whether the tanglegram it belongs to must be in $\mathcal{H}$ or not.

\begin{definition}\label{tangleclassreconstructable}
Let $n$ be a positive integer, and
$\D$  be a set (resp. multiset) of size $n-1$ tanglegrams.
We denote by $\mathbb{T}[\D]$ 
the set of size $n$ tanglegrams with deck (resp. multideck) $\D$. For tanglegram class $\mathcal{H}$ and positive integer $n$, we say that $\mathcal{H}$ is \emph{deck-decidable} (resp. \emph{multideck-decidable}) for size $n$ if for every set $\D$ (resp. multiset) of size $n-1$ tanglegrams with $\mathbb{T}(\mathcal{D}) \neq \emptyset$, either $\mathbb{T}[\D)]\subseteq\mathcal{H}$ or $\mathbb{T}[\D]\cap\mathcal{H}=\emptyset$, and we can decide from $\D$ which is the case.
\end{definition}

\begin{definition}
    Let $\mathcal{H}$ be a tanglegram class and $n$ be a positive integer.
    The tanglegrams of $\mathcal{H}$ of size $n$ are \emph{reconstructable} from their deck (resp. multideck) if $\mathcal{H}$ is deck-decidable (resp. multideck-decidable) for size $n$ and for every $\T \in \mathcal{H}$ of size $n$, $\T$ is reconstructable from its deck (resp. multideck). 
\end{definition}

Note that $\D(\T)$ (respectively $\Dm(\T)$) includes all elements of $\D(L)$ and $\D(R)$ (respectively $\Dm(L)$ and $\Dm(R)$). The following lemma is therefore immediate from Theorem~\ref{thm:tree_reconstruction}. 
\begin{lemma}\label{lem:reduce}
Let $\T=(L,R,\sigma)$ be a tanglegram. If $\size{\T}\ge 6$, then $L$ and $R$ can be reconstructed from $\D(\T)$, and if $\size{\T}=5$, then $L$ and $R$ can be reconstructed from $\Dm(\T)$. Consequently, the tanglegram class containing all tanglegrams
where at least one of the left or right tanglegram is a caterpillar is deck-decidable for size at least $6$.
\end{lemma}

Thus, our algorithmic approach reduces the tanglegram reconstruction problem to reconstructing the set $\sigma$ of matching edges from $\Dm(\T)$.

\section{Types of trees}\label{sec:types}
In this section, we partition the set of rooted binary trees into three types. This partition will be used in later sections to prove our main theorem.

\begin{definition}
Let $i\in\mathbb{N}$ and $T_1,T_2$ be rooted binary trees. We define the operation $(T_1\oplus)^i  T_2$  as follows:
$(T_1\oplus)^0T_2=T_2$ and for $i>0$, $(T_1\oplus)^i(T_2)=T_1\oplus((T_1\oplus)^{i-1}T_2)$.
\end{definition}

\begin{definition} Let $T$ be a rooted binary tree. Then $T$ is \emph{strippable} if $T=C_1$  or $T=C_1\oplus T'$ for some tree $T'$. Otherwise, $T$ is \emph{non-strippable}.
The \emph{stripping} of $T$ is defined as follows: If $T=C_n$, its stripping is $(n,\emptyset)$. Otherwise,
the stripping of $T$ is $(i,Q)$, where $Q$ is a non-strippable tree, $i$ is a nonnegative integer, and $T=(C_1\oplus)^iQ$. If $T=C_n$, all of its leaves are \emph{strippable}, otherwise, the leaves not in $Q$ are \emph{strippable} in $T$. Note that $i=0$ indicates that $T=T'$, and that $T$ is non-strippable (see, e.g., Fig.~\ref{fig:size5trees}).
\end{definition}

Note that if $(i,Q)$ is a stripping of a tree $T$, then $\size{T}=i+\size{Q}$ (where we use $\size{\emptyset}=0$). Moreover, if $i\ge 1$, then one of the following holds:  $i=n$ and $T=C_n$; or $\size{Q}\ge 4$.
When $T$ is strippable, we will often label the strippable leaves according to their distance from the root (the only time this is ambiguous is when $T=C_n$).

\begin{definition}\label{def:dtr}
    Let $T$ be a strippable tree with stripping $(i,Q)$. 
    A labeling of the strippable leaves with subscripted symbols $\{v_1,\ldots,v_i\}$ is a \emph{distance-to-root labeling} if for all $k:1\le k\le \min(i,n-2)$,
    $v_k$ is distance $k$ from $r_T$ for $1\leq k\leq i$. When $i>n-2$, then $T=C_n$, and  the \emph{distance-to-root labeling} gives the remaining two vertices the labels $v_n$ and $v_{n-1}$ (note that these two vertices form a cherry).
\end{definition} 

Note that when $T=C_n$, the  labels $v_{n-1}$ and $v_n$ may be used in two ways (interchanging the labels between the two leaves involved) in a distance-to-root labeling. Throughout, we will often identify a leaf with its label. For example, if we say that $v_i$ is matched to some leaf $\ell$, we mean that the leaf labeled by $v_i$ is matched to $\ell$. Furthermore, we often simply refer to the label set $\{v_1, \ldots, v_i\}$ as the distance-to-root labeling. 

\begin{definition}
    Let $T$ be a rooted binary tree. We say $T$ is \emph{type 0} if it is a caterpillar, \emph{type 1} if it is strippable and not a caterpillar, and \emph{type 2} otherwise. Moreover, a type 2 tree is \emph{type 2a} if it is of the form $Q\oplus Q$ for some rooted binary tree $Q$, \emph{type $2b$} if it is of the form $Q\oplus Q'$ for some rooted binary tree $Q$ and $Q'\in\D(Q)$, and \emph{type 2c} otherwise.
\end{definition}

Notice that if $T$ is type  1 with stripping $(i,Q)$, then $\size{Q} \geq 4$ and $\size{T}=\size{Q}+i\ge 5$. If $T$ is type 2, then $T = Q_1 \oplus Q_2$ for some rooted binary trees $Q_1, Q_2$ with $\size{Q_1}, \size{Q_2} \geq 2$.

In Fig.~\ref{fig:size5trees}, we depict the three rooted binary trees of size 5. They are of types 0, 1, and 2b, respectively.

\begin{figure}[htbp]
    \centering
    \scalebox{0.75}{
    \begin{tikzpicture}
        \pic[xscale=-1] at (4,0) {tree5.1}; 
        \node at (2,-1) {$C_5$, stripping $(5,\emptyset)$.};
        \pic[xscale=-1] at (9,0) {tree5.2}; 
        \node at (7,-1) {Tree with stripping $(1,C_2\oplus C_2)$.};
       \pic[xscale=-1] at (14,0) {tree5.3}; 
        \node at (12,-1) {$T$, stripping $(0,T)$.};      
    \end{tikzpicture}}
    \caption{The three trees of size $5$, respectively of types 0, 1, and 2b.}
    \label{fig:size5trees}
\end{figure}

The following statements are obvious and are implicit in \cite{Clifton2025}:
\begin{lemma}\label{lm:type1removal}
 Let $T$ be a tree of type 1 with stripping $(i,Q)$, Then we have:
\begin{enumerate}[label={\upshape (\roman*)}]
  \item If $v$ is a strippable leaf, then $T-v$ has stripping $(i-1,Q)$.
  \item If $v$ is a leaf of $Q$, then $T-v$ has stripping $(i+j,Q')$, where $(j,Q')$ is the stripping of $Q-v$.
 \end{enumerate}
 In particular, cards in $\D(T)$ have at least $i-1$ strippable leaves, and at least one card in $\D(T)$ has exactly $i-1$ strippable leaves; consequently, $i$ can be determined from $\D(T)$. 
 \\
 Let $T'\in \D(T)$ with stripping $(j,Q')$. Then we have one of the following:
\begin{enumerate}[label={\upshape (\alph*)}]
     \item $j=i-1$ and $Q'=Q$;
     \item $j=i$, $Q'\in\D(Q)$, and $T'=T-v$ for some leaf $v$ in a pending subtree of $Q$ that is of size at least $3$;
     \item $j>i$, $Q=C_2\oplus (C_1\oplus )^{j-i-1}Q'$, and $v$ is a leaf of the pending subtree $C_2$ of $Q$.
 \end{enumerate}
 In particular, we can determine which leaves of $T'$ correspond to leaves of $Q$ and which leaves correspond to strippable leaves of $T$. Moreover, we
 can determine which leaves of $T$ correspond to which pending subtree of $Q$, unless $Q$ is of type 2a or $Q= Q_1 \oplus Q_2$ is of type 2b and $T'=C_1\oplus(Q_j\oplus Q_j)$ where $\size{Q_j}>\size{Q_{3-j}}$.
 \end{lemma}

 \begin{lemma}\label{lm:type2removal}
Let $T=T_1\oplus T_2$ be a tree of type 2, let $v$ be any leaf of $T$, and let $T'=T-v$. 
If $T$ is not type 2a, we can determine a unique $j \in \{1,2\}$ such that $v \in T_j$ and $T' = (T_j - v) \oplus T_{3-j}$. In particular, we can determine whether $v \in T_1$ or $v \in T_2$.  
If, in addition, $T'$ is not type 2a,  we can identify which pending subtree of $T'$ corresponds to $T_{3-j}$.
\end{lemma}

\section{Tanglegrams of small size}
\label{sec:small}

\begin{figure}[htbp]
\centering
\scalebox{0.7}{
\begin{tikzpicture}
\pic at (15,0){example_5_1}; 
\end{tikzpicture} \qquad \qquad
\begin{tikzpicture}
\pic at (25,0){example_5_2}; 
\end{tikzpicture}
}
\\
\vspace{5mm}
\scalebox{0.7}{
\begin{tikzpicture}
\pic at (-3,0) {example_5_card_1}; 
\end{tikzpicture} \qquad
\begin{tikzpicture}
\pic at (10,0) {example_5_card_2}; 
\end{tikzpicture}
}
\caption{Two size $5$ catergrams that have the same multideck, and the cards in their deck. A distance-to-root labeling of the left and right trees of size $5$ catergrams are given as well as all representations of the cards in the deck.}
\label{fig:example_5}
\end{figure}

Consider size $n$ rooted binary trees: When $n\le 3$, the only rooted binary tree is the caterpillar $C_n$, for $n=4$, we have two trees, $C_4$ and $C_2\oplus C_2$, and for $n=5$, we have three trees, $C_5$, $C_1\oplus(C_2\oplus C_2)$, and $C_2\oplus C_3$. The two binary trees of size $4$ have the same multidecks: $4$ copies of $C_3$. 

Following the terminology of \cite{antichain22}, we refer to tanglegrams of the form $(C_n,C_n,\sigma)$, where the left and right trees are both caterpillars, as \emph{catergrams}. Clearly, the deck of a catergram only contains catergrams (and for $n\le 3$, every size $n$ tanglegram is a catergram). If we denote the number of size $n$ catergrams by $c_n$, then $c_1=1$, and, as it is implicit in \cite{antichain22}, for $n\ge 2$, we have $c_n=(n-2)!+\frac{n!-2(n-2)!}{4}=(n-2)!\frac{n^2-n+2}{4}$. In particular, $c_2=1$, $c_3=2$, and $c_4=7$. 
Moreover, every catergram of size $n$ (and for $n=4$ every tanglegram) has a deck that contains only catergrams of size $n-1$. Thus, the multideck of a size $n$ catergram can be described by the multiplicities of the size $(n-1)$ catergrams in it (including $0$). Therefore, there are at most $\binom{n+c_n-1}{c_n-1}$ possible decks for size $n$ catergrams. There are $17$ size $4$ tanglegrams, $7$ of these are catergrams, and there are at most $\binom{4+c_3-1}{c_3-1}=5$ multidecks for size $4$ tanglegrams. Consequently, even size $4$ catergrams
are not reconstructable from their multideck.

Thus, the best we can hope for is reconstructing tanglegrams of size at least $5$, and by Lemma~\ref{lem:reduce} we can reconstruct their left and right trees from their multideck. However, we find that even catergrams of size $5$ are not reconstructable from their multidecks, as the example in Figure~\ref{fig:example_5} shows. Therefore the lower bound $6$ in Theorem~\ref{mainthm1} is best possible.

\section{Reconstruction of catergrams} \label{Sec:CatergramReconstruction}
Throughout the remaining three sections, we freely use the fact that if $\T = (L, R, \sigma)$ is a tanglegram of size $n \geq 6$, the deck $\D(\T)$ uniquely determines the trees $L$ and $R$. Consequently, as long as $\size{\T}\ge 6$, we may freely assume that we know $L$ and $R$, and only need to determine $\sigma$ from $\Dm(\T)$ (see Lemma~\ref{lem:reduce}).
Also, without loss of generality, we will assume that $L=C_n$ and use the distance-to-root labeling $\{v_1,\ldots,v_n\}$. Moreover, for cards $(C_{n-1},R',\sigma')$ we use subscripted symbols $\{v_1',\ldots,v_{n-1}'\}$ for the distance-to-root labeling of the left subtree. \medskip

In this section, we prove reconstruction for catergrams of size at least $6$.
In this case, we use the distance-to-root labeling  $\{u_1,\ldots,u_n\}$ for the right subtree $R=C_n$, and for cards $T'=(C_{n-1},C_{n-1},\sigma')$, corresponding subscripted symbols $\{u_1',\ldots,u_{n-1}'\}$ for the distance-to-root labeling of the right subtree.

We first show that the multideck determines whether $v_1$ is matched to the cherry of $R$ or not.
\begin{lemma}\label{lm:matchcherry}
      Let $\mathcal{T} = (C_n,C_n,\sigma)$ be a catergram of size $n \ge 6$.
      We can determine from $\Dm(\T)$ whether $v_1$ is matched to one of $u_n,u_{n-1}$ or not.
\end{lemma}
\begin{proof}
 If $v_1$ is adjacent to one of $u_n,u_{n-1}$, we may assume without loss of generality that the distance-to-root labeling is such that $v_1$ is adjacent to $u_n$. Observe that if $v_1$ is matched to $u_{n}$, then at least $n-1\ge 5$ cards in $\Dm(\mathcal{T})$ have $v_1'$ matched to a leaf in the cherry of $R'$. 
 If $v_1$ is matched to $u_{n-2}$, then precisely two or three cards (corresponding to the removal of $u_{n-1},u_{n}$, and possibly $v_1$) have $v_1'$ matched to the cherry of $R'$. If $v_1$ is matched to some $u_i$ with $i\leq n-3$, then at most one card (namely $\T-v_1$) has $v_1'$ matched to the cherry of $R'$. Hence, by counting cards in which $v_1'$ is matched to the cherry of $R'$, we can determine from $\mathcal{D}^{m}(\mathcal{T})$ if $v_1$ is matched to one of $u_{n},u_{n-1}$ or not.    
\end{proof}

We now handle the case in which $v_1$ is matched to the cherry of $R$.

\begin{lemma}\label{lem:catcathelper}
    Let $\mathcal{T} = (C_n,C_n,\sigma)$ be a catergram of size $n \ge 6$.
    If $\sigma$ matches $v_1$ with a leaf of the cherry of the right tree, then $\mathcal{T}$ is reconstructable from the multideck $\mathcal{D}^{m}(\mathcal{T})$. 
\end{lemma}

\begin{proof}
There are various cases and subcases based on which vertices $v_2$ and potentially $v_3$ are matched to in $\T$. These are collected in Table~\ref{table:v1un_cases}.
  
 Observe that using $\Dm(\T)$, we can distinguish the case where $v_2$ is matched to $u_j$ with $j<n-2$, as in this case, there exists a unique card, $\T-v_1u_{n}$, where $v_1'$ is not matched to the cherry of $R'$, whereas if $v_2$ is matched to one of $u_{n-2},u_{n-1}$, then $v_1'$ is matched to the cherry of $R'$ in all cards $\T'$ of $\Dm(\T)$.
 
 If $v_2$ is matched to $u_j$ for some $j< n-2$,
 the card $\T-v_1u_{n}$ allows us to reconstruct $\T$.

\begin{table}[htbp]
 \caption{Case analysis for  $v_1u_n\in \sigma$. We assume without loss of generality that in the card $\T'=(L',R',\sigma')$, the distance-to-root labeling of $R'$ is chosen so that $v_1'$ is matched to $u_{n-1}'$ whenever it is matched to the cherry of $R'$.}
    \label{table:v1un_cases}
    \begin{tabular}{|c |c |c |c|}
        \hline 
        Case & Subcase & Card $\T'$ & Feature \\
        \hline 
        \multirow{3}{*}{$v_2u_{n-1}\in \sigma$} &  $v_3u_{n-2}\in \sigma$ & all & $v_1'u_{n-1}',v_2'u_{n-2}'\in \sigma'$ \\  \cline{2-4}
            & \multirow{2}{*}{$v_3u_{n-2}\notin \sigma$} & $\T-v_i,i\leq 2$ & $v_1'u_{n-1}'\in \sigma', \, v_2'u_{n-2}'\notin \sigma'$ \\ \cline{3-4}
            & & all remaining & $v_1'u_{n-1}',v_2'u_{n-2}'\in \sigma'$ \\ \hline
        \multirow{4}{*}{$v_2u_{n-2}\in \sigma$} & \multirow{2}{*}{$v_3u_{n-1}\in \sigma$} & $\T-v_i,i\leq 3$ & $v_1'u_{n-1}',v_2'u_{n-2}'\in \sigma'$ \\\cline{3-4}
            & & all remaining & $v_1'u_{n-1}'\in \sigma', \, v_2'u_{n-2}'\notin \sigma'$ \\ \cline{2-4}
            & \multirow{2}{*}{$v_3u_{n-1}\notin \sigma$} & $\T-u_{n-1}$ & $v_1'u_{n-1}',v_2'u_{n-2}'\in \sigma'$ \\\cline{3-4}
            & & all remaining & $v_1'u_n'\in \sigma',v_2'u_{n-2}'\notin \sigma'$ \\ \hline
        \multirow{2}{*}{all remaining} & \multirow{2}{*}{None} & $\T-v_1$ & $v_1'u_{n-1}'\notin \sigma'$\\ \cline{3-4} 
            & & all remaining & $v_1'u_{n-1}'\in \sigma'$ \\ \hline 
    \end{tabular}
\end{table}

Therefore we now assume that $v_2$ is matched to either $u_{n-1}$ or $u_{n-2}$.
For the remaining four subcases, observe that the number of cards where $v_1'$ and $v_2'$ are matched to the cherry of $R'$ is either $n$ (if $v_2$ is matched to $u_{n-1}$ and $v_3$ is matched to $u_{n-2}$), $n-2$ (if $v_2$ is matched to $u_{n-1}$ but $v_3$ is not matched to $u_{n-2}$), $3$ (if $v_2$ is matched to $u_{n-2}$ and $v_3$ is matched to $u_{n-1}$), or $1$ (if $v_2$ is matched to $u_{n-2}$ but $v_3$ is not matched to $u_{n-1}$). As $n\geq 6$, these are all distinct, so $\Dm(\T)$ allows us to distinguish these four subcases. It now suffices to show that we can reconstruct $\T$ in each of them.

In the case where $v_2u_{n-1},v_3u_{n-2}\in \sigma$, there is a maximal sequence $v_1,\ldots,v_{m}$ of vertices satisfying $v_iu_{n-i+1}\in \sigma$ for all $1\leq i\leq m$, where $m\geq 3$ (see Figure~\ref{fig:subcase-maximalsequence}). 

\begin{figure}[htbp]
\centering
\begin{tikzpicture}
\pic[scale = 0.25, xscale = 1.7] at (-3,0){both_caterpillars}; 
\pic[scale = 0.25, xscale = 1.7] at (3,0) {both_caterpillars_card_1};
\end{tikzpicture}
\caption{The tanglegram $\mathcal{T}$ in case $v_i$ is matched to $u_{n-i+1}$ ($1\leq i\leq m$) and its cards $\mathcal{T}-v_i$, for $1\leq i \leq m$.}\label{fig:subcase-maximalsequence}
\end{figure}

If $m = n$, then all cards are identical. If $m < n$, then the $m$ cards $\T-v_i u_{n-i+1}$ with $1\leq i\leq m$ are identical with $v_1',\ldots,v_{m-1}'$ satisfying $v_i'$ being matched to $u_{n-i}'$, while $v_m'$ is not matched to $u_{n-m}'$. The other $n-m$ cards in $\Dm(\T)$ have $v_i'$ matched to $u_{n-i}'$ for $1\leq i\leq m$. Thus, we can determine $m$, and using one of the identical cards $\T-v_iu_{n-i+1}$ with $1\leq i\leq m$, we can reconstruct $\T$.

In the case where $v_2u_{n-1}\in \sigma$ and $v_3u_{n-2}\notin \sigma$, the two cards where $v_2'$ is not matched to a leaf in the cherry of $R'$ are $\T-v_1$ and $\T-v_2$. Furthermore, these two cards are identical, and either one of them can be used to reconstruct $\mathcal{T}$. 

In the case where $v_2u_{n-2},v_3u_{n-1}\in \sigma$, the three cards where $v_1'$ and $v_2'$ are matched to the cherry of $R'$ are all the same. These correspond to $\T-v_1,\T-v_2$, and $\T-v_3$. Using any of these three cards allows us to reconstruct $\T$. 

Finally, in the case where $v_2u_{n-2}\in \sigma$ and $v_3u_{n-1}\notin \sigma$, we have that $u_{n-1}$ is matched to some $v_k$ with $k>3$. There is a unique card, $\T-v_ku_{n-1}$, in which the cherry of $R'$ is matched to both $v_1'$ and $v_2'$. Among the remaining cards, there are $n-k$ cards (corresponding to  $\T - v_i$ with $i > k$) where $v_k'$ is matched to the cherry of $R'$ and $k-1$ cards (corresponding to $\T - v_i$ with $i < k$) where $v_{k-1}'$ is matched to the cherry of $R'$. Hence, we can identify $k$, and using the card $\T-v_ku_{n-1}$ allows us to reconstruct $\T$. 
\end{proof}

\begin{lemma}\label{thm:cat-cat}
 Let $\mathcal{T}=(C_n, C_n, \sigma)$ be a catergram of size $n\geq 6$. Then $\T$ is reconstructable from $\mathcal{D}^{m}(\T)$.
\end{lemma}

\begin{proof}
    First note that Lemma~\ref{lem:reduce} implies that the class of catergrams of size at least $6$ are deck-decidable.
    
    Assume $v_1$ is matched to $u_i$ and $v_2$ is matched to $u_j$ with $i \neq j$. 
    By Lemma~\ref{lm:matchcherry}, we can determine from $\Dm(\T)$ whether $i\ge n-1$ or $i<n-1$.
    By Lemma~\ref{lem:catcathelper}, $\T$ is reconstructable if $i \ge n-1$, so we will assume $i < n-1$. Additionally, if $j\in \{n-1,n\}$, we set $j=n-1$. 

    If $i=1$, then at least $n-1$ of the cards in $\Dm(\T)$ have $v_1'$ matched to $u_1'$. If instead $i\neq 1$, then there are at most two such cards (if $i=2$, then 
    $\T-u_1$, if $j=1$, then $\T-v_1$). Thus, from $\Dm(\T)$, we can determine whether or not $i = 1$.

    Suppose $i = 1$. Let $v_1,\ldots,v_m$ be a maximal sequence of vertices satisfying that $v_i$ is matched to $u_i$ for all $1\leq i\leq m$. If $m=n$, all cards are identical and the tanglegram is reconstructable. Since $m=n-1$ is not possible, we consider $m<n-1$. Then, in exactly $m$ identical cards (corresponding to $\mathcal{T} - v_i u_i$ with $1 \leq i \leq m$), we have $v_k'$ matched to $u_k'$ for $1\leq k\leq m-1$ and $v_{m}'$ not matched to $u_m'$. In the remaining cards, $v_k'$ and $u_k'$ are matched for $1\leq k\leq m$. Thus, we can identify $m$, and the $m$ identical cards resulting from $\T-v_i u_i $ with $1\leq i \leq m$ allow us to reconstruct $\T$. 

    Suppose instead $i\neq 1$. In the cards of $\Dm(\T)$, we consider which vertex is matched to $v_1'$. From the cards other than $\T-v_1 u_i$, we get $i-1$ cards with $v_1'$ matched to $u_{i-1}'$ (corresponding to $\T - u_k$ with $1 \leq k < i$) and $n-i$ cards with $v_1'$ matched to $u_i'$ (corresponding to $\T - u_k$ with $i < k \leq n$). Which vertex $v_1'$ is matched to on the remaining card, $\T-v_1 u_i$,  depends on the relationship of $i$ and $j$. The vertices matched to $v_1'$ and their multiplicities in $\Dm(\T)$ are collected in Table~\ref{tab:caterpillarcases}. 
    
    \begin{table}[htbp]
    \caption{For $1<i<n-2$, the vertices matched to $v_1'$ in $\Dm(T)$ with multiplicities.}
    \label{tab:caterpillarcases}
    \begin{tabular}{|c|c|c|c|}
    \hline 
    $j>i+1$ & $j=i+1$ & $j=i-1$ & $j<i-1$ \\ \hline 
    $ u_{j-1}' \times (1)$ & $u_i' \times (n-i+1)$ & $ u_i' \times (n-i)$ & $ u_i' \times (n-i)$  \\ 
     $ u_i' \times (n-i)$ & $ u_{i-1}' \times (i-1)$ & $ u_{i-1}' \times (i)$ & $ u_{i-1}' \times (i-1)$ \\ 
     $ u_{i-1}' \times (i-1)$ & & & $ u_j' \times (1)$ \\ \hline 
    \end{tabular}
    \end{table}

    From Table~\ref{tab:caterpillarcases}, if $v_1'$ is matched to exactly two $u_k'$, then $j=i+1$ or $j=i-1$ and the multiplicities allow us distinguish between these cases. If $j=i+1$, then there is a maximal sequence of vertices $v_1,\ldots,v_m$ satisfying that $v_k$ is matched to $u_{i+k-1}$ for $1 \leq k \leq m-1$. Hence, there are $m$ identical cards (corresponding to $\T - v_i$ with $1 \leq i \leq m$) with $v_k'$ matched to $u_{i+k-1}'$ for $1 \le k \le m-1$ and $v_m'$ not matched to $ u_{i+m-1}'$. However, in all other cards, $v_1',\ldots,v_m'$ are matched to a set of $u$'s that are consecutively and increasingly indexed. Thus, we can identify $m$, and one of the $m$ identical cards resulting from removing $v_ku_{i+k-1}$ with $1\leq k\leq m$ allows us to reconstruct $\T$. If instead $j = i-1$, then a similar argument with $v_k u_{i-k+1} \in \sigma$ for $1\leq k\leq m$ shows that $\T$ is reconstructable. 

    Finally, we consider when $v_1'$ is matched to $u_k'$ for three different values of $k$. If these indices are non-consecutive, then we can easily identify both $i$ and the card $\T-v_1 u_i$. In this case, $\T$ is reconstructable from this card. If the indices are consecutive, then $j=i+2$ or $j=i-2$ and we consider the multiplicity of the cards in which $v_1'$ is matched to $u_k'$ with $k$ maximum. In the former case, there is one such card and in the latter case, there are $n-i \ge 2$ such cards (since $i \leq n-1$). Hence, we can distinguish between these cases, identify both $i$ and the card $\T-v_1 u_i$, and reconstruct $\T$. 
\end{proof}

\section{Caterpillar tanglegrams with a type 1 left or right tree} \label{Sec:Type1Reconstruction}

If the left or right tree of the caterpillar tanglegram is a type 1 tree, it is enough to consider the case when the left tree is $C_n$ and the right tree is type 1. We will use the distance-to-root labeling $\{u_1,\ldots,u_i\}$ for the strippable leaves of $R$,
and the corresponding distance to root labeling $\{u_1'\ldots,u'_{i'}\}$ for the strippable leaves of the right subtree of a card $(C_{n-1},R',\sigma')$ whenever the right tree $R'$ of the card is type 1.

\begin{lemma}\label{lem:type1-cherrywhere}
    Let $\T=(C_n, R, \sigma)$ be a tanglegram of size $n \geq 6$ such that $R$ is of type 1 with stripping $(i,Q)$. Then, $\D^m(\T)$ allows us to determine whether both, none, or exactly one of the two leaves of the cherry of the left tree $C_n$ are matched to a leaf from $Q$.
\end{lemma}

\begin{proof}
    Note that by Lemma~\ref{lm:type1removal}, for any $\T'=(L',R',\sigma')$ we can determine which leaves of $R'$ correspond to leaves of $Q$, and which leaves correspond to strippable leaves of $R$. For the sake of brevity, we refer to the leaves of $R'$ corresponding to leaves of $Q$ as leaves of $Q$, and if a leaf of $L'$ is matched to a leaf of $Q$ in $R'$, we say it is matched to $Q$.
    
    As Table~\ref{table:cat-type1-cherrymatch} shows, $\D^m(\T)$ allows us to determine whether both, none, or exactly one of $v_n$ and $v_{n-1}$ are matched to a leaf from $Q$.

    \begin{table}[htbp]
    \caption{Case analysis for the number of cards in $\D^m(\T)$ with the property that both of $v'_{n-1}$ and $v'_{n-2}$, exactly one of them, or none of them is matched to a leaf of $Q$ depending on whether both of $v_n$ and $v_{n-1}$, exactly one of them, or none of them, is matched to a leaf of $Q$ in $\T$.}
    \label{table:cat-type1-cherrymatch}
    \begin{tabular}{|>{\centering\arraybackslash}m{0.2\linewidth}| >{\centering\arraybackslash}m{0.2\linewidth}| >{\centering\arraybackslash}m{0.2\linewidth}| >{\centering\arraybackslash}m{0.2\linewidth}|}
    \hline 
     Case & number of cards with both $v'_{n-1}$ and $v'_{n-2}$ matched to $Q$ & number of cards with exactly one of  $v'_{n-1}$ and $v'_{n-2}$ matched to $Q$&  number of cards with none of  $v'_{n-1}$ and $v'_{n-2}$ matched to $Q$\\ \hline 
     both $v_n$ and $v_{n-1}$ matched to $Q$ & $\geq n-2 \geq 4$ & $\leq 2$ & -- \\  \hline
     exactly one of $v_n$ and $v_{n-1}$ matched to $Q$ & $\leq 1$ &  $\geq n-2 \geq 4$ &  $\leq 1$ \\ \hline
     none of $v_n$ and $v_{n-1}$ matched to $Q$ & -- &  $\leq 2$ & $\geq n-2 \geq 4$\\ \hline
    \end{tabular}
    \end{table}
\end{proof}

\begin{lemma}\label{lem:type1-cherrysame}
    Let $\T=(C_n, R, \sigma)$ be a tanglegram of size $n \geq 6$ such that $R$ is of type 1 with stripping $(i,Q)$ and such that either both or none of the leaves of the cherry of the left tree $C_n$ are matched to a leaf from $Q$. Then $\T$ is reconstructable from $\Dm(\T)$.
\end{lemma}

\begin{proof}
Let $X$ be the set of leaves of the left tree $C_n$ of $\T$ that are matched to $Q$, and $Y$ be the set of leaves of $C_n$ that are matched to the strippable leaves of $R$.
    Define the sequence $1\le j_1<j_2<\ldots<j_{\size{Q}}\le n$ such that
    $X=\{v_{j_q}: 1\le q\le\size{Q}\}$.

Assume first that both $v_n$ and $v_{n-1}$ are matched to $Q$. Let $s,t$ be such that the $s+1$ leaves in the sequence $v_n, v_{n-1}, \ldots, v_{n-s}$ are matched to $Q$, whereas the $t-s$ leaves in the sequence $v_{n-s-1}, \ldots, v_{n-t}$ are matched to strippable leaves of $R$, and either $t=n-1$ of the leaf $v_{n-t-1}$ is matched to $Q$.
Note that the sequences are defined such that they form a consecutive sequence in the distance-to-root labeling, and our assumption gives $1\le s\le n-i-1$ 
        and $s+1\le t\le n-1$.

        On each card $\T'=(C_{n-1},R',\sigma')\in\Dm(\T)$ at least one of the leaves of the  cherry of the left tree $C_{n-1}$ is matched to $Q$. For each, we choose a distance-to-root labeling $\{v_1',v_2'\ldots,v_{n-1}'\}$ so that $v'_{n-1}$ is matched to $Q$.
        
        In this case, there are precisely $s+1\ge 2$ cards $\T'\in\Dm(\T)$ (corresponding to the removal of $v_n, v_{n-1}, \ldots, v_{n-s}$, respectively) such that the $s\ge 1$ leaves in the sequence $v'_{n-1},v'_{n-2},\ldots,v'_{n-s}$ are matched to $Q$, while the $t-s$ leaves in the sequence $v'_{n-s-1},\ldots,v'_{n-t}$ are not matched to $Q$ and either $t=n-1$ or $v_{n-t-1}$ is matched to $Q$ again. On the remaining $n-s-1 \geq 1$ cards (corresponding to the removal of $v_{n-s-1}, \ldots, v_1$), the $s+1$ leaves in the sequence $v'_{n-1}, \ldots, v'_{n-s-1}$ are matched to $Q$. Thus, we can determine $s$ and $t$. 

        Consider a card $\T'= (C_{n-1},R',\sigma')$ where $v'_{n-s-1}$ is not matched to $Q$. Then $\T'$ is missing one of $v_n,v_{n-1},\ldots, v_{n-s}$ and a leaf of $Q$. Thus, from $\T'$, we can determine $X$ and $Y$. Furthermore, for $j,m$ with $1\le j\le n-s-1$ and $1\le m\le i$, we have $v'_ju'_m\in\sigma'$ precisely when $v_{j}u_m\in\sigma$ so we can determine the matching between $Y$ and the strippable leaves of $R$.

        Now we know $X$ and $j_1,\ldots,j_{\size{Q}}$. By our assumption, $j_{\size{Q}}=n$ and $j_{\size{Q}-1}=n-1$.
         Choose a card $\T'=(C_{n-1},R',\sigma')$ 
         that misses one of the strippable leaves of $R$. Then $\T'$ includes all leaves of $Q$ and $X$. 
         The matching between $X$ and $Q$ can be determined from $\T'$ as follows.
         Let $X'$ be the set of leaves in $C_{n-1}$ that are matched to $Q$, and let $j'_1<j'_2<\ldots<j'_{\size{Q}}$ such that
         $X'=\{v'_{j'_q}: 1\le q\le\size{Q}\}$. 
         For each $q$ with $1\le q\le\size{Q}-2$ and each leaf $w$ of $Q$ we have
         that $v'_{j'_q}w\in\sigma'$ precisely when
         $v_{j_q}w\in\sigma$. Moreover, if $x,y$ are such that $v'_{n-2}x,v'_{n-1}y\in\sigma'$, then either 
         $v_{n-1}x,v_{n}y\in\sigma$ or $v_{n-1}y,v_nx\in\sigma$. As $v_{n-1},v_n$ are leaves of a cherry, both of these result in the same tanglegram, therefore we can reconstruct $\T$.

Assume next that neither  $v_n$ nor $v_{n-1}$ is matched to $Q$. In this case,  $v_n$ and $v_{n-1}$ are both matched to strippable leaves of $R$. 
        Choose $s,t$ to be maximal such that 
        the $s+1$ leaves in the sequence $v_{n}, v_{n-1}, \ldots, v_{n-s}$ are matched to strippable leaves of $R$ and the 
        $t-s \geq 1$ leaves in the sequence
        $v_{n-s-1}, \ldots, v_{n-t}$ are matched to $Q$. We can now repeat the argument used previously with the roles of $Q$ and the strippable leaves of $R$ interchanged. 
\end{proof}

\begin{lemma}\label{lem:cat-type1}
    Let $\T=(C_n, R, \sigma)$ be a tanglegram of size $n \geq 6$ such that $R$ is type 1. Then, $\T$ is reconstructable from the multideck $\Dm(\T)$.
\end{lemma}

\begin{proof}  
By Lemma~\ref{lem:reduce}, we can reconstruct the left tree $C_n$ and the right tree $R$ from $\D(\T)$; moreover, the class of tanglegrams of size at least $6$ where the right tree is $C_n$ and the left tree is type 1 is deck-decidable. Since $R$ is type 1, for some $i>0$ and $Q$ a non-strippable tree, we have that the stripping of $R$ is $(i,Q)$.
By Lemma~\ref{lem:type1-cherrywhere}, $\Dm(\T)$ allows us to determine how many of the leaves of the cherry of the left tree $C_n$ are matched to $Q$.
If both or neither of the leaves are matched to $Q$, we are done by Lemma~\ref{lem:type1-cherrysame}.

So assume that exactly one of $v_n,v_{n-1}$ is matched to $Q$ and the other is matched to a strippable leaf of $R$.
Without loss of generality, let the distance-to-root labeling be such that $v_n$ is matched to $Q$ and $v_{n-1}$ is matched to a strippable leaf of $R$. As before, let $X$ be the set of leaves of the left tree $C_n$ of $\T$ that are matched to $Q$, and let $Y$ be the set of leaves of $C_n$ that are matched to the strippable leaves of $R$.
Define $1\le j_1<j_2<\ldots<j_{\size{Q}}\le n$ such that
    $X=\{v_{j_q}: 1\le q\le\size{Q}\}$.        
    
Notice that there is precisely one card $\T^*=(C_{n-1},R^*,\sigma^*)$
        in the deck of $\T$ (either $\T-v_n$ or $\T-v_{n-1}$) , where the leaves of the cherry in the left tree are both matched to $Q$ or they are both matched to strippable leaves of $R$; moreover,
         $v_{n-2}$ is matched to $Q$ by $\sigma$  precisely when in $\T^*$ both leaves of the cherry of the left tree are matched to $Q$.

         Thus, $\T^*$ allows us to determine  $X$ and $Y$.
         
         Without loss of generality we choose the distance-to-root labeling of the left tree $C_{n-1}$ for any 
         $\T'\ne\T^*$ such that $v'_{n-1}$ is matched to $Q$. 
         
         Choose a card $\T'$ different from $\T^*$ from which a leaf of $Q$ has been removed. As $\size{Q}\ge 4$, this is possible. 
         As in this card $v'_{n-1}$ is matched to $Q$, this card determines
         the matching between $Y$ and the strippable leaves of $R$. We only need to determine the matching between $X$ and the leaves of $Q$. 
         
 	If  $i\ge 2$, choose a card $\T'$ different from $\T^*$ from which a strippable leaf of $R$ has been removed. As $i\ge 2$, this is possible. As in this card $v'_{n-1}$ is matched to $Q$ and  $v'_{n-2}$ is matched to a strippable leaf of $R$, the matching between $X$ and $Q$ can be determined.

        Thus, we are left with the case  when $i=1$. Therefore $R=C_1\oplus Q$ with $\size{Q}\ge 5$. 
            As $v_{n-1}$ is matched to the (unique) strippable leaf of $R$, $v_{n-2}$ is matched to $Q$. 
                                    
            In this case, note that $\T^*= \T - v_{n-1}$ allows us to determine not just $X$ and $Y$, but also both the matching of the leaves of $X-\{v_n,v_{n-2}\}$ to the leaves of $Q$, and the two vertices $x,y$ of $Q$ to which $v_n$ and $v_{n-2}$ are matched. 
            To reconstruct $\T$, we only need to determine which of $x,y$ is matched with $v_n$, or realize that both possibilities result in the same tanglegram.
		
	 We write $Q=Q_1\oplus Q_2$, where $2\le\size{Q_1}\le\size{Q_2}$.
	 
	 We consider the cases as follows.
	 \begin{enumerate}
				            
          \item Vertices $x$ and $y$ are in the same pending subtree $Q_k$ of $Q$:

          In this case in any card $\T'$ different from $\T^*$, the vertex $v'_{n-1}$ is matched to $Q_k$ and $v'_{n-2}$ is matched to the strippable leaf of $R$. Consider a card $\T' \neq \T^*$ where all vertices of the pending subtree of $Q$ to which $v'_{n-1}$ is matched are present.

	In $\T'$, $v'_{n-1}$ corresponds to $v_n$ 
    and $v'_{n-3}$ corresponds to $v_{n-2}$ and hence we can determine $\T$.   

	\item Vertices $x$ and $y$ are in different pending subtrees of $Q$:

Consider cards that contain the strippable leaf of $R$ (i.e. cards different from $\T^*$); there are $n-1$ of them.
In exactly one of these cards (either
$\T-v_n$ or $\T-v_{n-2}$), vertices $v'_{n-1}$ and $v'_{n-3}$ are matched to the same pending subtree of $Q$. In the remaining $n-2$ cards, 
$v'_{n-1}$ and $v'_{n-3}$ are matched to different pending subtrees of $Q$. Moreover, of these, $n-3\ge 3$ contain both $v_n$ and $v_{n-2}$ from the original graph while one contains exactly one of $v_n,v_{n-2}$. We further consider cases. 
\begin{enumerate}
\item $Q$ is type 2c:

By Lemma~\ref{lm:type2removal}, we can identify vertices corresponding to $Q_1$ and $Q_2$ in any card.
 
Consider the cards different from $\T^*$ in which $v'_{n-1}$ and $v'_{n-3}$ are matched to different pending subtrees of $Q$, and choose $k$ such that that $v'_{n-1}$ is matched to $Q_k$ in the majority of the cards. Then $v_n$ is matched to the vertex from $x,y$ in $Q_k$ and $\T$ is determined.  

\item $Q$ is type 2b:

In this case, as $\size{Q}$ is odd and $i=1$, $n$ is even. Moreover, $Q_1\in\D(Q_2)$. 
By Lemma~\ref{lm:type2removal}, we can identify for any card whether we removed a vertex from $Q_1$ or $Q_2$ to obtain this card.
Moreover, unless the card is type 2a (and we removed a vertex from $Q_2$), we can identify the
vertices corresponding to $Q_1$ and $Q_2$ in the card. 

For $j\in\{1,2\}$, we denote by $X_j$ the set of vertices in $X-\{v_n,v_{n-2}\}$ that are matched to $Q_j$. Note that $\T^*$ determines $X_1$ and $X_2$.
As $v_n,v_{n-2}$ are matched to different pending subtrees of $Q$, 
$|X_2|=\frac{n-2}{2}\ge 2$  and $|X_1|=\frac{n-4}{2}\ge 1$. 

We have the following possibilities:

\begin{enumerate}
\item $v_{n-3}\in X_2$: 

This means that exactly half of the leaves $v_1,v_2,\ldots,v_{n-4}$ are matched to $Q_1$ and the other half are matched to $Q_2$.

There are precisely $|X_1|\ge 1$ cards (of the form $\T-v_j$, $v_j\in X_1=X_1-\{v_{n-3}\}$) such that the number of leaves among $v'_1,\ldots,v'_{n-4}$ that are matched to
$Q_1$ is $|X_1|-1$; in the rest of the cards the number of such leaves is $|X_1|$. 
Take a card in which the number of leaves among $v'_1,\ldots,v'_{n-4}$ that are matched to
$Q_1$ is $|X_1|-1$. 
As in this card we have that for $t\in \{1,2,3\}$, vertex $v'_{n-t}$ corresponds to $v_{n+1-t}$, $\T$ is determined.

\item $v_{n-3}\in X_1$ and $|X_1|\ge 2$: 

Among the $\size{Q_1}=|X_1|+1\ge 3$ cards that have a leaf from $Q_1$ removed, one does not contain both $v_n$ and $v_{n-2}$, the remaining $|X_1|\ge2$ cards do.
Without loss of generality, $x$ is the vertex of $Q_2$. Then $x$ is matched to $v_n$ precisely when in the majority of the cards which have a leaf from $Q_1$ removed, the vertex $v'_{n-1}$ is matched to $x$. Thus, we can determine $\T$ from $\Dm(\T)$.

\item $X_1=\{v_{n-3}\}$:

In this case $\size{Q_1}=2$, $\size{Q_2}=3$, $n=6$, and $R=C_1\oplus(C_2\oplus C_3)$.
Hence, $v_{n-3}=v_3$ is matched to $Q_1=C_2$.

Without loss of generality $x$ is in $Q_1$ and $y$ is in $Q_2$.

We need to determine whether $v_6$ is matched to $x$ or $y$, or, in other words, whether $x$ is matched to $v_6$ or $v_4$.

\begin{figure}[http]
\begin{tikzpicture}[scale=.7]
\draw[fill=gray!20] (2,3)--(3,2)--(2,1)--(2,3);   
\node at (2.5,2) {$Q_2$}; 
\node[fill=black,rectangle,inner sep=2pt,label=above:{$v_5$}] (v5) at (0,6) {};      
\node[fill=black,rectangle,inner sep=2pt,label=above:{$v_6$}] (v6) at (0,5) {};      
\node[fill=black,rectangle,inner sep=2pt,label=above:{$v_4$}] (v4) at (0,4) {}; 
\node[fill=black,rectangle,inner sep=2pt,label=above:{$v_3$}] (v3) at (0,3) {};  
\node[fill=black,rectangle,inner sep=2pt,label=above:{$v_2$}] (v2) at (0,2) {};      
\node[fill=black,rectangle,inner sep=2pt,label=above:{$v_1$}] (v1) at (0,1) {};      
\node[fill=black,rectangle,inner sep=2pt] (u5) at (2,6) {};      
\node[fill=black,rectangle,inner sep=2pt,label=above:$x$] (x) at (2,5) {};      
\node[fill=black,rectangle,inner sep=2pt] (u3) at (2,4) {}; 
\node[fill=black,rectangle,inner sep=2pt,label=left:$y$] (y) at (2,3) {};  
\draw (v5)--(-.5,5.5)--(v6);
\draw (-.5,5.5)--(-1,5)--(v4);
\draw (-1,5)--(-1.5,4.5)--(v3);
\draw (-1.5,4.5)--(-2,4)--(v2);
\draw (-2,4)--(-2.5,3.5)--(v1);
\draw (x)--(2.5,4.5)--(u3);
\draw (3,2)--(4,3) --(2.5,4.5);
\draw (u5)--(4.5,3.5) --(4,3);
\node[fill=black,circle,inner sep=1pt]  at (-.5,5.5) {};
\node[fill=black,circle,inner sep=1pt]  at (-1,5) {};
\node[fill=black,circle,inner sep=1pt]  at (-1.5,4.5) {};
\node[fill=black,circle,inner sep=1pt]  at (-2,4) {};
\node[fill=black,circle,inner sep=1pt]  at (-2.5,3.5) {};
\node[fill=black,circle,inner sep=1pt]  at (2.5,4.5) {};
\node[fill=black,circle,inner sep=1pt]  at (3,2) {};
\node[fill=black,circle,inner sep=1pt]  at (4,3) {};
\node[fill=black,circle,inner sep=1pt]  at (4.5,3.5) {};
\draw[dashed] (v6)--(x);
\draw[dashed] (v4)--(y);
\draw[dashed] (v5)--(u5);
\draw[dashed] (v3)--(u3);
\draw[dashed] (v2)--(2,2);
\draw[dashed] (v1)--(2,1);
\node at (1,.5) {$\T_1$};
\end{tikzpicture}
\qquad
\begin{tikzpicture}[scale=.7]
\draw[fill=gray!20] (2,3)--(3,2)--(2,1)--(2,3);   
\node at (2.5,2) {$Q_2$}; 
\node[fill=black,rectangle,inner sep=2pt,label=above:{$v_5$}] (v5) at (0,6) {};      
\node[fill=black,rectangle,inner sep=2pt,label=above:{$v_6$}] (v6) at (0,5) {};      
\node[fill=black,rectangle,inner sep=2pt,label=above:{$v_4$}] (v4) at (0,4) {}; 
\node[fill=black,rectangle,inner sep=2pt,label=above:{$v_3$}] (v3) at (0,3) {};  
\node[fill=black,rectangle,inner sep=2pt,label=above:{$v_2$}] (v2) at (0,2) {};      
\node[fill=black,rectangle,inner sep=2pt,label=above:{$v_1$}] (v1) at (0,1) {};      
\node[fill=black,rectangle,inner sep=2pt] (u5) at (2,6) {};      
\node[fill=black,rectangle,inner sep=2pt,label=above:$x$] (x) at (2,5) {};      
\node[fill=black,rectangle,inner sep=2pt] (u3) at (2,4) {}; 
\node[fill=black,rectangle,inner sep=2pt,label=left:$y$] (y) at (2,3) {};  
\draw (v5)--(-.5,5.5)--(v6);
\draw (-.5,5.5)--(-1,5)--(v4);
\draw (-1,5)--(-1.5,4.5)--(v3);
\draw (-1.5,4.5)--(-2,4)--(v2);
\draw (-2,4)--(-2.5,3.5)--(v1);
\draw (x)--(2.5,4.5)--(u3);
\draw (3,2)--(4,3) --(2.5,4.5);
\draw (u5)--(4.5,3.5) --(4,3);
\node[fill=black,circle,inner sep=1pt]  at (-.5,5.5) {};
\node[fill=black,circle,inner sep=1pt]  at (-1,5) {};
\node[fill=black,circle,inner sep=1pt]  at (-1.5,4.5) {};
\node[fill=black,circle,inner sep=1pt]  at (-2,4) {};
\node[fill=black,circle,inner sep=1pt]  at (-2.5,3.5) {};
\node[fill=black,circle,inner sep=1pt]  at (2.5,4.5) {};
\node[fill=black,circle,inner sep=1pt]  at (3,2) {};
\node[fill=black,circle,inner sep=1pt]  at (4,3) {};
\node[fill=black,circle,inner sep=1pt]  at (4.5,3.5) {};
\draw[dashed] (v6)--(y);
\draw[dashed] (v4)--(x);
\draw[dashed] (v5)--(u5);
\draw[dashed] (v3)--(u3);
\draw[dashed] (v2)--(2,2);
\draw[dashed] (v1)--(2,1);
\node at (1,.5) {$\T_2$};
\end{tikzpicture}
\vskip 10pt
\begin{tikzpicture}[scale=.7]
\draw[fill=gray!20] (2,3)--(3,2)--(2,1)--(2,3);   
\node at (2.5,2) {$Q_2$}; 
\node[fill=black,rectangle,inner sep=2pt,label=above:{$v_5$}] (v5) at (0,5) {};        
\node[fill=black,rectangle,inner sep=2pt] (v4) at (0,4) {}; 
\node[fill=black,rectangle,inner sep=2pt] (v3) at (0,3) {};  
\node[fill=black,rectangle,inner sep=2pt,label=above:{$v_2$}] (v2) at (0,2) {};      
\node[fill=black,rectangle,inner sep=2pt,label=above:{$v_1$}] (v1) at (0,1) {};      
\node[fill=black,rectangle,inner sep=2pt] (u5) at (2,5) {};      
\node[fill=black,rectangle,inner sep=2pt] (u3) at (2,4) {}; 
\node[fill=black,rectangle,inner sep=2pt,label=left:$y$] (y) at (2,3) {};  
\draw (v4)--(-.5,4.5)--(v5);
\draw (v3)--(-1,4) --(-.5,4.5);
\draw (-1,4)--(-1.5,3.5)--(v2);
\draw (-1.5,3.5)--(-2,3)--(v1);
\draw (3,2)--(3.5,2.5)--(u3);
\draw (u5)--(4,3) --(3.5,2.5);
\node[fill=black,circle,inner sep=1pt]  at (-.5,4.5) {};
\node[fill=black,circle,inner sep=1pt]  at (-1,4) {};
\node[fill=black,circle,inner sep=1pt]  at (-1.5,3.5) {};
\node[fill=black,circle,inner sep=1pt]  at (-2,3) {};
\node[fill=black,circle,inner sep=1pt]  at (3,2) {};
\node[fill=black,circle,inner sep=1pt]  at (4,3) {};
\node[fill=black,circle,inner sep=1pt]  at (3.5,2.5) {};
\draw[dashed] (v4)--(y);
\draw[dashed] (v5)--(u5);
\draw[dashed] (v3)--(u3);
\draw[dashed] (v2)--(2,2);
\draw[dashed] (v1)--(2,1);
\node at (1,.5) {$\T_1-v_6=\T_2-v_4=\T_2-v_3$};
\end{tikzpicture}
\qquad
\begin{tikzpicture}[scale=.7]
\draw[fill=gray!20] (2,3)--(3,2)--(2,1)--(2,3);   
\node at (2.5,2) {$Q_2$}; 
\node[fill=black,rectangle,inner sep=2pt,label=above:{$v_5$}] (v5) at (0,5) {};      
\node[fill=black,rectangle,inner sep=2pt,label=above:{$v_6$}] (v6) at (0,4) {};      
\node[fill=black,rectangle,inner sep=2pt,label=above:{$v_4$}] (v4) at (0,3) {}; 
\node[fill=black,rectangle,inner sep=2pt,label=above:{$v_2$}] (v2) at (0,2) {};      
\node[fill=black,rectangle,inner sep=2pt,label=above:{$v_1$}] (v1) at (0,1) {};      
\node[fill=black,rectangle,inner sep=2pt] (u5) at (2,5) {};      
\node[fill=black,rectangle,inner sep=2pt,label=above:$x$] (x) at (2,4) {};      
\node[fill=black,rectangle,inner sep=2pt,label=left:$y$] (y) at (2,3) {};  
\draw (v5)--(-.5,4.5)--(v6);
\draw (-.5,4.5)--(-1,4)--(v4);
\draw (-1,4)--(-1.5,3.5)--(v2);
\draw (-1.5,3.5)--(-2,3)--(v1);
\draw (3,2)--(3.5,2.5) --(x);
\draw (u5)--(4,3) --(3.5,2.5);
\draw[dashed] (v6)--(x);
\draw[dashed] (v4)--(y);
\draw[dashed] (v5)--(u5);
\draw[dashed] (v2)--(2,2);
\draw[dashed] (v1)--(2,1);
\node[fill=black,circle,inner sep=1pt]  at (-.5,4.5) {};
\node[fill=black,circle,inner sep=1pt]  at (-1,4) {};
\node[fill=black,circle,inner sep=1pt]  at (-1.5,3.5) {};
\node[fill=black,circle,inner sep=1pt]  at (-2,3) {};
\node[fill=black,circle,inner sep=1pt]  at (3,2) {};
\node[fill=black,circle,inner sep=1pt]  at (4,3) {};
\node[fill=black,circle,inner sep=1pt]  at (3.5,2.5) {};
\node at (1,.5) {$\T_1-v_3$};
\end{tikzpicture}
\caption{Case analysis: when $x$ is matched to $v_6$ ($\T_1$) and when $x$ is matched to $v_4$ ($\T_2$).}
\label{fig:c6c5card}
\end{figure}

There are precisely two trees in $\D(\T)$ with $R'=C_5$: $\T-v_3$ and $\T-x$. 

If $x$ is matched to $v_6$, then
$\T-x=\T-v_6$ and $\T-v_3$ are different cards in the deck, because in
$\T-v_6$ we have $v'_4$ matched to $u'_j$ for some $j\ge 3$ while in
$\T-v_3$ we have $v'_4$ matched to $u'_2$. 

On the other hand, if $x$ is matched to $v_4$, then the cards $\T-v_3$ and $\T-x=\T-v_4$ are the same card (with multiplicity $2$ in the multideck) (see Figure~\ref{fig:c6c5card}).

Therefore, we can determine whether $v_6$ is matched to $x$ or $y$, and
$\T$ is reconstructable from $\Dm(\T)$.
\end{enumerate}

\item $Q$ is type 2a: 

Since $\size{Q}$ is even, $n$ is odd, and $n\ge 7$.

Among $v_1,v_2,\ldots,v_{n-3}$ exactly $\frac{n-3}{2}\ge 2$ leaves are matched to $Q_1$ and the same number of leaves are matched to $Q_2$. Without loss of generality, $x$ is in the pending subtree matched to $v_1$, and this pending subtree is labeled $Q_1$. 

For $j\in\{1,2\}$, we denote by $X_j$ the set of vertices in $X-\{v_n,v_{n-2}\}$ that are matched to $Q_j$. Note that $\T^*$ determines $X_1$ and $X_2$ in this case as well
(since it determines the partition $\{X_1,X_2\}$ and $v_1 \in X_1$).
Choose $t$ such that in the sequence $v_1,\ldots,v_t$, all leaves are matched to $Q_1$ but $v_{t+1}$ is matched to $Q_2$. We have $1\le t\le\frac{n-3}{2}<n-3$.
As we know $X_1$, we know the value of $t$. 

If $t>1$, then
in every card, we have $v'_1,\ldots,v'_{t-1}$ matched
to $Q_1$. Choose a card in which $v'_1$ and $v'_t$ are matched to different pending subtrees of $Q'$. There are exactly $t\ge 2$ such cards (corresponding to $\T-v_p$ for $1\le p\le t<n-3$). Then $v_n$ is matched to $x$ precisely when $v'_{n-1}$ and $v'_1$ match to the same pending subtree of $Q'$ in these cards.

If $t=1$, there is precisely one card in which
$v'_1$ and $v'_2$ are matched to the same pending subtree of $Q'$. This card is $\T-v_1$ if $v_3$ is matched to $Q_2$ and $\T-v_2$ if $v_3$ is matched to $Q_1.$ As $n-3\ge 4$, and we know $X_1$ and $X_2$, we know whether $v_1$ and $v_3$ are matched to the same pending subtree of $Q$. Therefore this card determines $\T$.

\end{enumerate}
\end{enumerate}
\end{proof}

\section{Caterpillar tanglegrams with type 2 left or right tree} \label{Sec:Type2Reconstruction}
As before, it is enough to consider tanglegrams of the form $(C_n,R,\sigma)$ where $R$ is type $2$. 

Throughout this section, we will assume that $\{v_1,\ldots,v_n\}$ is a distance-to-root labeling of the leaves of the left tree $C_n$ and, on each card, we let $\{v_1',\ldots,v_{n-1}'\}$ be a distance-to-root labeling of the leaves of the left tree $C_{n-1}$. If $v_n,v_{n-1}$ are matched to different pending subtrees of $R$, without loss of generality, we choose the labeling such that $v_{n-1}$ and $v_{n-2}$ are matched to the same subtree of $R$.

We write $R = R_1 \oplus R_2$ with $\size{R_2}\ge \size{R_1} \geq 2$. If $\size{R_1}=\size{R_2}$, without loss of generality, we choose the labeling of the pending
subtrees of $R$ so that $v_n$ is matched to $R_1$. 

We set $i$ such that $v_n$ is matched to $R_i$. Note that $i=2$ precisely when the two pending subtrees of $R$ are not the same size, and $v_n$ is matched to the larger pending subtree.

For $j\in\{1,2\}$, let $X_j$ denote the set of leaves in $C_n$ that are matched to the pending subtree $R_j$, and $k_{1,j}<k_{2,j}\cdots<k_{|X_j|,j}$ be the sequence
for which $X_j=\{v_{k_{p,j}}: 1\le p\le |X_j|\}$. 

For a card $(C_{n-1},R',\sigma')$, we will write $R'=R'_1\oplus R'_2$, where we assume that $R'_1$ corresponds to $R_1$ and $R'_2$ corresponds to $R_2$. Note that the indexing $R'_1$ and $R'_2$ 
    may not be easily recognized in a card, so we need to pay attention to that issue throughout the proofs.

\begin{lemma}\label{lem:type2-cherrywhere}
    Let $\T=(C_n, R, \sigma)$ be a tanglegram of size $n \geq 6$ such that $R$ is type 2. Then $\Dm(\T)$ determines the value of $i$ and whether the leaves of the cherry of the left tree are matched to the same pending subtree of $R$ or not. 
\end{lemma}

\begin{proof}

    We first argue that we can determine from $\Dm(\T)$ whether $v_n$ and $v_{n-1}$ are matched to the same pending subtree of $R$ or not.  
    As $\size{R_1}\ge 2$,  in any card $\T'=(C_{n-1},R',\sigma')$, either $R'$ is of type 2, or $R_1=C_2$ and $R'=C_1\oplus R_2$. 
    
    If $v_n$ and $v_{n-1}$ are matched to the same pending subtree $R_i$ of $R$, then on at least $n-2 \geq 4$ cards, leaves $v'_{n-2}$ and $v'_{n-1}$ are matched to the same pending subtree of $R'$.
    If $v_n$ and $v_{n-1}$ are not matched to the same pending subtree of $R$, then on precisely one card, leaves $v'_{n-2}$ and $v'_{n-1}$ are matched to the same pending subtree of $R'$ (either $\T-v_n$ or $\T-v_{n-1}$). Thus, we can distinguish these two cases. 
    
    If $\size{R_1}=\size{R_2}$, then $i=1$, so we can determine $i$. 
    
    If $\size{R_1}\ne\size{R_2}$, then $R$ is not type 2a. Consider the cards in which $v'_{n-1},v'_{n-2}$ are matched to the same subtree of $R'$ (there is at least one such card), and take one in which this subtree has the largest size. 
    We have the following cases:
    \begin{enumerate}
    \item $R$ is type 2c or $R$ is type 2b and the two subtrees of $R'$ are of different sizes:
    
    We can identify $R'_1,R'_2$ correctly on this card.

    Choose $j$ such that $v'_n,v'_{n-1}$ are matched to the same subtree $R'_j$ of $R'$ in the card.
    If $v_n,v_{n-1}$ are matched to the same subtree of $R$, then $i=j$. Otherwise $i=3-j$.
    \item $R$ is type 2b and the two subtrees of $R'$ are of the same size:
    
    In this case the card was obtained by removing a leaf from $R_2$. 
         
     If $v_n,v_{n-1}$ are matched to the same pending subtree of $R$, then the maximality condition ensures this card contains all leaves of this pending subtree and $i=1$.
    
    If $v_n,v_{n-1}$ are matched to different subtrees of $R$, then this card is $\T-v_n$,  so $i=2$.
    \end{enumerate}
    
    Thus, we can determine $i$ from $\Dm(\T)$.

  
 \end{proof}

\begin{lemma}\label{lem:type2-cherrysame}
    Let $\T=(C_n, R, \sigma)$ be a tanglegram of size $n \geq 6$ such that $R$ is type 2. If the leaves of the cherry of the left tree are matched to the same pending subtree of $R$, then $\T$ is reconstructable from $\Dm(\T)$. 
\end{lemma}

\begin{proof}
As $n\ge 6$, we can reconstruct $C_n$ and $R$ from $\D(\T)$ (Lemma~\ref{lem:reduce}).
By Lemma ~\ref{lem:type2-cherrywhere}, $\Dm(\T)$ determines that $v_n,v_{n-1}$ are matched to the same subtree $R_i$ of $R$, and also the index $i$ of this subtree.
Therefore the class of tanglegrams of size at least $6$ with left tree caterpillar and right tree type 2 tree and the leaves of the cherry are matched to the same pending subtree of the right tree is a deck-decidable class.

In addition, in any card in which the leaves $v'_{n-1},v'_{n-2}$ of the left tree $C_{n-1}$ are matched to the same subtree of $R'$, we can determine the indices $R'_1$, $R'_2$ correctly, as we know the value of $i$ and $R'_i$ is the subtree of $R'$ that both $v'_{n-1}$ and $v'_{n-2}$ are matched to.

Choose $s,t$ such that the $s+1$ leaves in the sequence  $v_n, v_{n-1}, \ldots, v_{n-s}$ are all matched to the same pending subtree $R_i$ of $R$, 
 the $t-s$ leaves in the sequence $v_{n-s-1}, \ldots, v_{n-t}$ are matched to $R_{3-i}$, and either $t=n-1$ or $v_{n-t-1}$ is matched to $R_i$. We have
        $1\le s\le n-3$ and $s+1 \leq t \leq n-1$. 
        
        We start with showing that we can determine $s,t$ from $\Dm(\T)$.

        If $s=1$, then in exactly two cards  ($\T-v_n$ and $\T-v_{n-1}$), $v'_{n-1}$ and $v'_{n-2}$ are matched to different pending subtrees of $R'$.
        If $s\ge 2$, then in all cards, we have  $v'_{n-1}$ and $v'_{n-2}$  
        matched to the same pending subtree of $R'$. Moreover, when $s\ge 2$, in exactly $s+1$ cards (of the form $\T-v_j$ where $n-s\le j\le n$),
        $v'_{n-1}$ and $v'_{n-s-1}$ are not matched to the 
        same pending subtree of $R'$.  Thus, $\Dm(\T)$ determines the value of $s$. Moreover, if $s\ge 2$, we can identify $R'_1$ and $R'_2$ in every card, and
        if $s=1$, we can identify $R'_1$ and $R'_2$ in every card different from $\T-v_n$ and $\T-v_{n-1}$.
                
        If $s\ge 2$, then in all of the $s+1$ cards where $v'_{n-1}$ and $v'_{n-s-1}$ are matched to different subtrees of $R'$, every leaf in the sequence
        $v'_{n-s-1},\ldots,v'_{n-t}$ is matched to the same pending subtree of $R'$ while either $t=n-1$ or $v'_{n-t-1}$ is matched to a different pending 
        subtree, so we can determine $t$ from $\Dm(\T)$.

        If $s=1$ and $t=1$, then there is precisely one card ($\T-v_{n-2}$) in which $v'_{n-1},v'_{n-2}$, and $v'_{n-3}$ are all matched to the same pending subtree of $R'$. If $s=1$ and $t\ge 2$, then in none of the cards we have $v'_{n-1},v'_{n-2}$, and $v'_{n-3}$ all matched to the same pending subtree of $R'$.
        
       If $s=1$ and $t\ge 2$, there are precisely two cards $(\T-v_n$ and $\T-v_{n-1}$) such that $v'_{n-1}$ and $v'_{n-2}$ are not matched to the same pending subtree of $R'$. Consider these two cards only. If $t=n-1$, then $v'_1,v'_2,\ldots,v'_{n-3}$ are all matched to the same pending subtree of $R'$. If $t<n-1$, then  $v'_{n-3}$ and $v'_{n-t-1}$ are matched to different pending subtrees of $R'$, but for each $j$ with $3\le j\le t$, $v'_{n-3}$ and $v'_{n-j}$ are matched to the same pending subtree of $R'$.
        Therefore, when $s=1$, $\Dm(\T)$ determines the value of $t$. 
        
        Moreover, if $s=1$, then in the cards where $v'_{n-1}$ and $v'_{n-2}$ are not matched to the same subtree of $R'$ (which are $\T-v_n$ and $\T-v_{n-1}$)
        , leaf $v'_{n-3}$ is matched to $R'_i$ when $t=1$ and it is matched to $R'_{3-i}$ when $t\ge 2$. Therefore, even when $s=1$, we can identify the proper indices for $R'_1$ and $R'_2$ in all cards. Without loss of generality, when $s=1$ and $v'_{n-1},v'_{n-2}$ are matched to different subtrees, we choose the distance-to-root labeling of the left tree $C_{n-1}$ such that $v'_{n-1}$ is matched to $R_i$, so $v'_{n-2}=v'_{n-s-1}$ is matched to $R_{3-i}$.

        Now choose  a card in which $v'_{n-s-1}$ is matched to $R_{3-i}$. This means that a leaf in $R_i$ that was matched to one of the $s+1$ leaves $v_n,v_{n-1},\ldots,v_{n-s}$ was removed but all leaves from $R_{3-i}$ are present. This allows us to determine both $X_1$ and $X_2$, and, as $v_n,v_{n-1}\notin X_{3-i}$, this  card also determines the matching between $X_{3-i}$ and $R_{3-i}$.

        Next choose a card in which $v'_{n-s-1}$ 
        is matched to $R_i$, implying that a leaf from $R_{3-i}$ was removed to obtain $R'_{3-i}$ (as we know $R_1,R_2$ and $i$, and can identify $R'_{1},R'_2$ correctly, such a card can be found). Since we know the set $X_i$, and $v_n,v_{n-1}\in X_i$, this card allows us to determine the matching between $X_i$ and $R_i$, while for the two vertices matched to $v_n$ and $v_{n-1}$, we may not know which vertex is matched to $v_n$ and which to $v_{n-1}$. However, the two possible matchings give the same unlabeled tanglegram. Therefore, $\T$ is reconstructable from $\Dm(\T)$.
\end{proof}

\begin{lemma}\label{lem:cat-type2}
    Let $\T=(C_n, R, \sigma)$ be a tanglegram of size $n \geq 6$. If $R$ is type 2, $\T$ is reconstructable from $\Dm(\T)$.
\end{lemma}

\begin{proof}
As $n\ge 6$, we can reconstruct $C_n$ and $R$ from $\D(\T)$ by Lemma~\ref{lem:reduce}; and the class of tanglegrams of size at least $6$  where the left tree is a caterpillar and the right tree is type 2 is deck-decidable.
Lemma~\ref{lem:type2-cherrywhere} gives that $\Dm(\T)$ determines whether $v_n,v_{n-1}$ are matched to the same subtree of $R$, and also the index $i$.
If $v_n,v_{n-1}$ are matched to the same subtree of $R$, we are done by Lemma~\ref{lem:type2-cherrysame}.  

Assume that $v_n$ and $v_{n-1}$ are matched to different pending subtrees of $R$. 
       Recall that we choose a distance-to-root labeling such that if $v_n$ and $v_{n-1}$ are matched to different subtrees of $R$, then $v_{n-2},v_{n-1}$ are matched to
       the same subtree $R_{3-i}$.
       Let $t$ be such that the $t$ leaves in the sequence $v_{n-1}, v_{n-2}, \ldots, v_{n-t}$ are all matched to $R_{3-i}$ while $v_{n-t-1}$ is matched to $R_i$. We have $2 \leq t \leq n-2$. 

        We have exactly one card (namely $\T-v_n$) in which $v'_{n-1},v'_{n-2}$ are matched to the same pending subtree of $R'$. Consider this card:
        all leaves in the sequence $v'_{n-1},v'_{n-2},\ldots,v'_{n-t}$ are matched to $R'_{3-i}$ and $v'_{n-t-1}$ is matched to $R'_i$.
        In any other card,  $v'_{n-1}$ and $v'_{n-2}$ are matched to different pending subtrees. Therefore, we can determine $t$ from $\Dm(\T)$.

        Take the unique card in which $v'_{n-1},v'_{n-2}$ are matched to the same subtree $R'_{3-i}$. As this card is $\T-v_n$, $v_n\in X_i$, and $v_{n-1}\in X_{3-i}$,  we can determine $X_1$ and $X_2$, and the pending subtree $R_{3-i}$ of $R$ (uniquely, unless $R$ is type 2a).  
        
        Assume first that $t\ge 3$. Consider only the $n-1$ cards in which $v'_{n-1}$ and $v'_{n-2}$ are matched to different subtrees of $R'$, and without loss of generality, use the labeling in which $v'_{n-2}$ and $v'_{n-3}$ are matched to the same subtree of $R'$. As $t\ge 3$, this ensures that $v'_{n-1}$ corresponds to the vertex $v_n$ and $v'_{n-2},v'_{n-3}$ are matched to the pending subtree $R'_{3-i}$. This ensures that the pending subtrees $R'_1$, $R'_2$ are correctly identified in these cards. 
        
        Find a card in which the size of the subtree $v'_{n-2}$ is matched to is $\size{R_{3-i}}=|X_{3-i}|$. As $v_n\notin X_{3-i}$, and this card is of the form $\T-v_j$
        for some $v_j\in X_i-\{v_n\}$, this card allows us to determine the matching between vertices of $X_{3-i}$ and $R_{3-i}$.

        Now find a card in which the size of the subtree $v'_{n-1}$ is matched to is $\size{R_{i}}=|X_{i}|$. As $v_{n-1},v_{n-2}\notin X_{i}$, and this card is of the form $\T-v_j$
        for some $v_j\in X_{3-i}$, this card allows us to determine the matching between vertices of $X_{i}$ and $R_{i}$. 
        Thus, we can reconstruct $\T$ when $t\ge 3$.

        Assume in what remains that $t=2$. We have the following cases: 
  
 \begin{enumerate}
	\item $R$ is type 2c: 
	
	Since we know $R_1,R_2$, and $i$, and we can correctly identify $R'_1$ and $R'_2$ in every card, in every card different from $\T-v_n$, we can choose the distance-to-root labeling such that $v'_{n-1}$ is matched to $R'_i$ and $v'_{n-2}$ is matched to $R'_{3-i}$.  Consider the cards different from $\T-v_n$ and take one in which a leaf from $R_{3-i}$ was removed. As in this card, $v'_{n-1}$ is matched to $R_i$ and  $v'_{n-2}$ is matched to $R'_{3-i}$, this gives us the matching between $X_i$ and $R_i$.
	Take another card that is not $\T-v_n$, in which a leaf from $R_i$ was removed --- this gives us the matching between $X_{3-i}$ and $R_{3-i}$. This determines $\T$.

        \item $R$ is type 2b:
        
        We know $i$, $R_1,R_2$, and we can identify $R_1'$, $R_2'$ correctly in every card in which $R'$ is not type 2a,
        Moreover, $\size{R_1}\ge 3$, $\size{R_2}=\size{R_1}+1\ge 4$, and if $R'$ is type 2a, then a leaf from $R_2$ was removed to obtain $R'$. 
        
        Take a card from which a leaf from $R_1$ was removed (so $R'_1$ and $R'_2$ can be identified) and $v'_{n-1},v'_{n-2}$ are matched to different pending subtrees of $R'$.
        This card determines the matching between $X_2$ and $R_2$.

        Now consider $Z=X_2-\{v_1,v_n,v_{n-2},v_{n-1}\}$. As $|X_2|=\size{R_2}\ge 4$ and only one of $v_n,v_{n-1}$ is in $X_2$, $Z\ne\emptyset$. For any $v_j\in Z$, 
        card $\T-v_j$ has $v'_{n-1}$ and $v'_{n-2}$ matched to different pending subtrees of $R'$. Without loss of generality, choose the labeling such that
        $v'_{n-2}$ and $v'_{n-3}$ are matched to the same pending subtree of $R'$.
        Then $v'_{n-1}$ and $v'_1$ are matched to the same pending subtree of $R'$ precisely when $v_1\in X_i$.

        Consider any card where a leaf from $R_2$ has been removed, $v'_{n-1},v'_{n-2}$ are matched to different pending subtrees of $R'$, and $v'_1$ is matched to the opposite subtree as $v'_{n-3}$ exactly if $v_1 \in X_i$. We may choose the labeling so that $v'_{n-2}$ and $v'_{n-3}$ are matched to the same pending subtree and so $v'_1$ and $v'_{n-1}$ are matched to the same subtree exactly when $v_1 \in X_i$. 
        This card is of the form $\T-v_p$ for some $p\ne n$. 
        
        We will first show that if
        $v_p\notin Z$, then $p=1$ and $v_1,v_2$ are matched to the same subtree of $R$.
        
        If $v_p\notin Z$, then $p\in\{n-1,n-2,1\}$. 
        
        Assume $p\in\{n-1,n-2\}$. Since $t=2$, we have that $v'_{n-3}$ corresponds to $v_{n-3}$, $v'_{n-1}$ corresponds to one
        of $v_{n-1},v_{n-2}$ (the other is $v_p$), $v'_{n-2}$ corresponds to $v_{n}$, and $v'_1$ corresponds to $v_1$. However, this contradicts the fact that $v'_{n-1}$ and $v'_1$ match to the same subtree of $R'$ precisely when $v_1\in X_i$. 
        Therefore, if $v_p\notin Z$, then $p=1$.

        Assume $p=1$. The card we choose is $
        \T-v_1$, so for
        each $j\in \{1, \ldots, n-1\}$, $v'_{n-j}$ corresponds to $v_{n+1-j}$, and $v_n,v_1$ are matched to the same subtree of $R$ precisely when $v_n,v_2$ are matched to the same subtree of $R$. Therefore, $v_1,v_2$ are matched to the same subtree of $R$.

        Thus, whether or not $p \in Z$, this card allows us to determine the matching between $R_1$ and $X_1$. Thus, $\T$ is reconstructable from $\Dm(\T)$.

        \item $R$ is type 2a:   
        
        In this case $i=1$, $n$ is even, $|X_1|=|X_2|=\size{R_1}=\size{R_2}\ge 3$,
        $v_n,v_{n-3}\in X_1$ and $v_{n-1},v_{n-2}\in X_2$.
        Note that for $j\in\{1,2\}$, a matching between $X_j$ and $R_j$ is an ordering $w_1,\ldots,w_{\frac{n}{2}}$ of the leaves of $R_j$ such that for each $q\in\left \lbrace 1, \ldots,\frac{n}{2}\right \rbrace$, we have that $v_{k_{q,j}}$ is matched to $w_q$.

        Consider the $n-1$ cards in which $v'_{n-1},v'_{n-2}$ are matched to different subtrees of $R'$ (so $v_n$ was not removed). Each of these give an ordering of the leaves of the larger pending subtree of $R'$, so one of the two possible matchings from $X_1$ to $R_1$ or $X_2$ to $R_2$.

        If all $n-1$ of these orderings are the same (up to isomorphism of the subtree), $\T$ is determined.

        Otherwise, one ordering appears $\frac{n}{2}-1$ times, and the other $\frac{n}{2}$ times. As $i=1$, $v_n\in X_1$, and $\T-v_n$ is not among the $n-1$ cards considered, the ordering that appears fewer times gives the matching between $X_2$ and $R_2$. $\T$ is determined.
        \end{enumerate}
\end{proof}

\section{Conclusion}\label{sec:conclusion}

Symmetry arguments and Lemmas~\ref{thm:cat-cat}, \ref{lem:cat-type1}, and \ref{lem:cat-type2} establish our main theorem, which we re-state for the reader's convenience.

\begin{reptheorem}{mainthm1}
Let $\mathcal{T}=(L,R,\sigma)$ be a tanglegram of size $n\geq 6$. If $L$ or $R$ is a caterpillar, then $T$ is reconstructable from the multideck $\Dm(\T)$.
\end{reptheorem}

We note that Theorem~\ref{mainthm1} establishes the reconstruction of caterpillar tanglegrams of size $n \geq 6$ from the \emph{multi}deck $\Dm(\T)$. Indeed, our proofs rely on the multiplicities of cards in several instances. However, since trees of size $n \geq 6$ are reconstructable from their decks (Lemma~\ref{lem:reduce}), it is an interesting open question to investigate whether decks also determine sufficiently large tanglegrams of some tanglegram class. 

The original question --- whether tanglegrams are reconstructable for sufficiently large size --- remains open. Interestingly, not all tanglegrams of size $6$ are reconstructable. An example of a pair of tanglegrams of size $6$ that have the same multideck is shown in Figure~\ref{fig:example_6}.  Using computer search \cite{github}, we found $4$ pairs of size-$6$ tanglegrams that have the same multideck. In each of these tanglegrams, the left and right trees were a $C_3\oplus C_3$, and either $C_2\oplus (C_2\oplus C_2)$ or $(C_1\oplus)^2(C_2\oplus C_2)$.

\begin{figure}
\centering
\scalebox{0.7}{
\begin{tikzpicture}
\pic at (15,0){example_6_1}; 
\end{tikzpicture} \qquad \qquad
\begin{tikzpicture}
\pic at (25,0){example_6_2}; 
\end{tikzpicture}
}
\\
\vspace{5mm}
\scalebox{0.6}{
\begin{tikzpicture}
\pic at (-6,0) {example_6_card_1}; 
\end{tikzpicture} \quad
\begin{tikzpicture}
\pic at (6,0) {example_6_card_2}; 
\end{tikzpicture} \quad 
\begin{tikzpicture}
\pic at (18,0) {example_6_card_3}; 
\end{tikzpicture}} 
\\
\vspace{5mm}
\scalebox{0.6}{
\begin{tikzpicture}
\pic at (0,0) {example_6_card_4}; 
\end{tikzpicture} \quad 
\begin{tikzpicture}
\pic at (12,0) {example_6_card_5}; 
\end{tikzpicture}
}
\caption{Two non-isomorphic size $6$ tanglegrams that have the same multideck, and the cards in their deck. A labeling of the left and right trees of size $6$ tanglegrams are given as well as all representations of the cards in the deck.}
\label{fig:example_6}
\end{figure}

In general, we ask the following: 
\begin{problem} 
 Given a tanglegram class $\mathcal{H}$, determine whether sufficiently large tanglegrams in the class are reconstructable from their deck or multideck. If yes, what is the best threshold for ``sufficiently large"?
\end{problem}

The requirement of reconstructibility here includes being able to identify membership in $\mathcal{H}$ from the deck or multideck (see Definition~\ref{tangleclassreconstructable}). Theorem~\ref{mainthm1} is of this type: 
$\mathcal{H}$ is the set of tanglegrams where at least one of the left and right tree is a caterpillar. In this case (and in any case when membership in $\mathcal{H}$ is determined by identifying left and right trees) being able to identify the membership relation from the multideck or the deck is not a significant requirement (see Lemma~\ref{lem:reduce}). 

A class of tanglegrams where left and right trees alone do not determine membership in the class is the class of planar tanglegrams, that is, tanglegrams that have a layout without any crossing edges. Czabarka, Wagner and Sz\'ekely~\cite{Czabarka2018} showed a tanglegram analogy of Kuratowski's theorem: every nonplanar tanglegram contains one of two size $4$ subtanglegrams. This result implies that given a deck of any tanglegram with size different from $4$, one can identify from its deck whether it is planar or not.
Interestingly, none of the tanglegrams
among the size $6$ tanglegrams that are not reconstructable from their multideck are planar, therefore it is possible that all size at least $6$ planar tanglegrams are reconstructable from their multideck. 
The authors are preparing a manuscript on the reconstructibility of planar tanglegrams.

\section*{Acknowledgements}
This project was initiated when the authors participated in the  Mathematics Research Community workshop ``Trees in Many Contexts" (Beaver Hollow, NY), organized by the American Mathematical Society with the  support of the National Science Foundation under Grant Number DMS-1641020.  The work on this project was continued when EC, LS, and KW were in residence at the Institute for Computational and Experimental Research in Mathematics of Brown University  (Providence, RI), during the ``Theory, Methods, and Applications of Quantitative Phylogenomics" semester program supported by the National Science Foundation under Grant Number DMS-1929284. We thank two anonymous reviewers for their valuable feedback on an earlier version of this manuscript.

\bibliographystyle{plain}
\bibliography{bibliography}

\end{document}